\documentclass{article}
\usepackage[utf8]{inputenc}
\usepackage[margin=1in]{geometry}
\usepackage{amsmath}
\usepackage{amsfonts}
\usepackage{amssymb}
\usepackage{tikz-cd}
\usepackage{color,soul}
\usepackage{array}
\usepackage{zref-savepos}
\usepackage{amsthm}
\usepackage{multirow, bigstrut}
\usepackage[font={small,it}]{caption}
\usepackage{tikz}
\usepackage{bm}
\usepackage{titling}
\usetikzlibrary{arrows}

\usepackage{hyperref}
\hypersetup{
    colorlinks=true,
    linkcolor=blue,
    filecolor=magenta,  
    urlcolor=cyan,
}\usepackage[capitalize, nameinlink]{cleveref}
\usepackage{thmtools, thm-restate}

\usetikzlibrary{calc}
\usepackage{comment}
\usepackage{enumerate}
\usepackage{tikz}
\usetikzlibrary{shapes.geometric}

\newtheorem{theorem}{Theorem}[section]
\newtheorem*{theorem*}{Theorem}
\newtheorem{lemma}[theorem]{Lemma}
\newtheorem{proposition}[theorem]{Proposition}
\newtheorem{corollary}[theorem]{Corollary}

\theoremstyle{definition}
\newtheorem{definition}[theorem]{Definition}

\newtheorem{example}[theorem]{Example}

\theoremstyle{plain}

\newcommand{\R}{\mathbb{R}}
\newcommand{\N}{\mathbb{N}}
\newcommand{\Z}{\mathbb{Z}}

\newcommand{\E}{\mathbb{E}}
\newcommand{\Prob}{\mathbb{P}}

\renewcommand{\P}[1]{{\mathbb{P}}\left[#1\right]}
\newcommand{\PP}[2]{{\mathbb{P}}_{#1}\left[#2\right]}
\def\eps{{\epsilon}}
\def\bbe{{\bf e}}
\def\bbf{{\bf f}}

\newcommand{\subLP}{P_{\mathrm{Sub}}}

\newcommand{\EE}[2]{{\mathbb{E}}_{#1}\left[#2\right]}

\DeclareMathOperator{\cpc}{Cap}

\DeclareMathOperator{\per}{per}

\DeclareMathOperator{\Mat}{Mat}

\DeclareMathOperator{\NHMat}{NHMat}
\DeclareMathOperator{\NHProd}{NHProd}
\DeclareMathOperator{\NHStab}{NHStab}

\usepackage[backend=bibtex, style=alphabetic, backref=true,maxbibnames=99, url=false]{biblatex} 
\title{From Trees to Polynomials and Back Again: \\ New Capacity Bounds with Applications to TSP}
\author{Leonid Gurvits, Nathan Klein, and Jonathan Leake}

\begin{document}


\maketitle

\begin{abstract}
    We give simply exponential lower bounds on the probabilities of a given strongly Rayleigh distribution, depending only on its expectation. This resolves a weak version of a problem left open by Karlin-Klein-Oveis~Gharan in their recent breakthrough work on metric TSP, and this resolution leads to a minor improvement of their approximation factor for metric TSP. Our results also allow for a more streamlined analysis of the algorithm.

    To achieve these new bounds, we build upon the work of Gurvits-Leake on the use of the productization technique for bounding the capacity of a real stable polynomial. This technique allows one to reduce certain inequalities for real stable polynomials to products of affine linear forms, which have an underlying matrix structure. In this paper, we push this technique further by characterizing the worst-case polynomials via bipartitioned forests. This rigid combinatorial structure yields a clean induction argument, which implies our stronger bounds.
    
    In general, we believe the results of this paper will lead to further improvement and simplification of the analysis of various combinatorial and probabilistic bounds and algorithms.
\end{abstract}


\section{Introduction}

The theory of real stable and log-concave polynomials has seen numerous applications in combinatorics and theoretical computer science (TCS). This includes bounds and approximation algorithms for various combinatorial quantities \cite{Gur07VdW,Barv16,AOG17,SV17,AMOGV18,AOGV18i,AASV21,BLP23}, proofs of long-standing log-concavity and sampling conjectures related to matroids \cite{AHK18,ALOGV18iii,ALOGV19ii,BH20}, proofs of the Kadison-Singer conjecture and generalizations \cite{MSS15,AOG14KS,Bran18KS}, an improved approximation factor for the traveling salesperson problem (TSP) \cite{OGSS11,KKO21,KKO23}, and many more. 
The power of these polynomial classes comes from two features: (1) their \emph{robustness}, shown in the fact that many natural operations preserve these log-concavity properties, and (2) their \emph{convex analytic properties}, which can be used to prove bounds and other analytic statements. The typical way these polynomials are utilized is by encoding combinatorial objects as real stable and log-concave polynomials, which essentially allows these operations and convexity properties to automatically transfer to the combinatorial objects. This idea, while simple, has led to important breakthroughs in combinatorics, TCS, and beyond.

For example, \cite{Gur07VdW} utilized real stable polynomials to give a new proof of the Van der Waerden conjecture on the permanent of a doubly stochastic matrix (originally due to \cite{Egor81,Fal81}). This proof led to a vast generalization of Van der Waerden bound, including an improved Schrijver's bound for regular bipartite graphs \cite{Gur07VdW}, an analogous bound for mixed discriminants \cite{Gur06MD}, and an analogous bound for mixed volumes that led to the development of strongly log-concave polynomials \cite{Gur09MV,Gur09}. One reason the original bound was historically so difficult to prove is a lack of a usable inductive structure coming from the matrices themselves. One of the key insights of the new proof was to use the simple inductive structure of real stable polynomials given by partial derivatives. By encoding the matrices as polynomials, the correct induction becomes straightforward, and the bound follows from a simple calculus argument.


More recently, the approximation factor improvement for the metric traveling salesperson problem (TSP) crucially utilized real stable polynomials \cite{KKO21,KKO23}. The idea is to encode certain discrete probability distributions related to spanning trees as real stable polynomials. The coefficients of these polynomials give probabilities of certain graph-theoretic events (e.g., the number of edges in a given spanning tree incident on a particular vertex), and analytic properties of real stable polynomials allow one to lower bound these probabilities. This in turn implies bounds on the expected cost of the output of a randomized algorithm for metric TSP.

In this paper, we improve upon the polynomial capacity bounds of \cite{GL21}, and our applications touch on the two problems discussed above. Specifically, we give:
\begin{enumerate}
    \item robust coefficient lower bounds for all (not necessarily homogeneous) real stable polynomials,
    \item simply exponential lower bounds on probabilities of strongly Rayleigh distributions (solving a weak version of an open problem of \cite{KKO21}), and
    \item a further improvement to the approximation factor for metric TSP (predicted by \cite{GL21}).
\end{enumerate}
Interestingly, our approach goes in the opposite direction to that discussed above. Our technical results answer questions regarding real stable polynomials, but to prove these results we use various graph and matrix structures inherent to the polynomials. In \cite{GL21}, this was seen in the ``productization'' technique: bounds on real stable polynomials were achieved by showing that the worst-case bounds come from polynomials associated to certain matrices. In this paper, we push this idea further by showing that these worst-case matrices are bipartite adjacency matrices of forests. This very rigid structure enables a clean induction argument, which implies stronger polynomial capacity bounds. These new bounds lead to the applications discussed above, with the strongest bounds implying the metric TSP improvement.






\section{Main Results} \label{sec:main-results}

We first state here our main technical results; see \cref{sec:tech_overview} for any undefined notation.

Our first main result is a non-homogeneous version of Theorem 2.1 of \cite{GL21} which implies robust coefficient lower bounds for all real stable polynomials as a direct corollary. Crucially, these bounds do not depend on the total degree of the polynomial. This was one of the main barriers to applying the results of \cite{GL21} to metric TSP.

\begin{theorem}[Main non-homogeneous capacity bound] \label{thm:main_nh}
    Let $p \in \R_{\geq 0}[x_1,\ldots,x_n]$ be a real stable polynomial in $n$ variables, and fix any $\bm\kappa \in \Z^n$ with non-negative entries.
    If $p(\bm{1}) = 1$ and $\|\bm\kappa - \nabla p(\bm{1})\|_1 < 1$, then
    \[
        \inf_{x_1,\ldots,x_n > 0} \frac{p(\bm{x})}{x_1^{\kappa_1} \cdots x_n^{\kappa_n}} \geq \left(1 - \|\bm\kappa - \nabla p(\bm{1})\|_1\right)^n.
    \]
    This bound is tight for any fixed $\bm\kappa$ with strictly positive entries.
\end{theorem}

\begin{corollary}[Main non-homogeneous coefficient bound] \label{cor:main_nh_coeff}
    Let $p \in \R_{\geq 0}[x_1,\ldots,x_n]$ be a real stable polynomial in $n$ variables, and fix any $\bm\kappa \in \Z^n$ with non-negative entries.
    If $p(\bm{1}) = 1$ and $\|\bm\kappa - \nabla p(\bm{1})\|_1 < 1$, then
    \[
        p_{\bm\kappa} \geq \left(\prod_{i=1}^n \frac{\kappa_i^{\kappa_i} e^{-\kappa_i}}{\kappa_i!}\right) (1 - \|\bm\kappa - \nabla p(\bm{1})\|_1)^n,
    \]
    where $p_{\bm\kappa}$ is the coefficient of $\bm{x}^{\bm\kappa}$ in $p$. The dependence on $(1 - \|\bm\kappa - \nabla p(\bm{1})\|_1)$ is tight for any fixed $\bm\kappa$ with strictly positive entries.
\end{corollary}

The above results\footnote{Note that these bounds already appear in the arXiv version of \cite{GL21}, but not in the STOC version.} are robust (i.e., resilient to $\ell^1$ perturbations) versions of the results utilized to bound various combinatorial and probabilistic quantities, as discussed above. That said, they are still not quite strong enough to imply an improvement to the metric TSP approximation factor. To obtain this improvement, we resolve a weak version of an open problem from \cite{KKO21}, which we discuss below. Stronger versions of \cref{thm:main_nh} and \cref{cor:main_nh_coeff} which imply this result can be found in \cref{sec:main_proofs}.




\subsection{Application: Minimum Permanent}

Before discussing the application to TSP, we first describe a different application of our bounds as a sort of prelude. It is at this point well-known that the permanent of any $n \times n$ doubly stochastic matrix is at least $\frac{n!}{n^n}$, and that $\frac{1}{n} \bm{1} \cdot \bm{1}^\top$ is the unique minimizer of the permanent over all doubly stochastic matrices. On the other hand, a similar tight lower bound with explicit minimizer is not known for sets of matrices with different row and column sums. The following then slightly extends what is known in the doubly stochastic case. See \cref{sec:unique-minimizers} for further details.

Given $\bm{c} \in \R_{\geq 0}^n$, let $\Mat_n(\bm{c})$ denote the set of $n \times n$ matrices with non-negative entries, row sums equal to $\bm{1}$, and columns sums equal to $\bm{c}$.

\begin{theorem} \label{thm:unique-minimizers-main}
    For all $n \geq 1$, there exists $\epsilon_n > 0$ such that if $\|\bm{c} - \bm{1}\|_1 < \epsilon_n$ then $\frac{1}{n} \bm{1} \cdot \bm{c}^\top$ is the unique minimizer of the permanent over $\Mat_n(\bm{c})$. Specifically, this holds if $\frac{c_1 c_2 \cdots c_n}{L(\bm{c})} < \frac{(n-2)^{n-2} n^{n-1}}{(n-1)^{2n-3}}$, where $L(\bm{c})$ is any lower bound on the capacity.
\end{theorem}

The proof utilizes \cref{thm:main_nh} or results of \cite{GL21} to guarantee all minimizers of the permanent lie in the relative interior of $\Mat_n(\bm{c})$. The symmetry and multilinearity of the permanent then imply the minimizer must be the rank-one matrix of $\Mat_n(\bm{c})$ described above. See Section \cref{sec:unique-minimizers} for the explicit value of $\epsilon_n$ and the details of the proof. It should be noted that in proving uniqueness, we were able to avoid usage of the conditions for equality in the Alexandrov-Fenchel inequalities.

On the other hand, when $\bm{c}$ is far from $\bm{1}$, the above result can be far from correct. Recall from \cite{GL21} that $\per(M) > 0$ for all $M \in \Mat_n(\bm{c})$ if and only if $\|\bm{c} - \bm{1}\|_1 < 2$.

\begin{proposition} \label{prop:minimizer-counter}
    For all $n$ large enough, there exists $\bm{c}$ such that $\|\bm{c} - \bm{1}\|_1 < 2$ and a sparse matrix $M \in \Mat_n(\bm{c})$ (with linearly many non-zero entries) which has smaller permanent than that of $\frac{1}{n} \bm{1} \cdot \bm{c}^\top$.
\end{proposition}

These two results suggest the complexity of the minimizer of the permanent, given the column sums $\bm{c}$. The coefficient bound given above in \cref{cor:main_nh_coeff} is then a lower bound which generalizes that of the permanent to coefficients of real stable polynomials. (Consider the coefficient $p_{\bm{1}}$ of $p(\bm{x}) = \prod_{i=1}^n \sum_{j=1}^n m_{ij} x_j$.) Thus \cref{cor:main_nh_coeff} can be seen as a sort of ``smoothing'' of the complexities that can occur for minima of the permanent and its generalizations.

Further, \cref{thm:unique-minimizers-main} can generalized to a permanent-like function on rectangular matrices using \cref{thm:main_nh}, and even beyond that to the mixed discriminant. (For the mixed discriminant, the row sum condition becomes a trace condition, and the column sum condition becomes an eigenvalue condition.) On the other hand, we can generalize the statement of \cref{thm:unique-minimizers-main} to coefficients of real stable polynomials in general, but we do not yet know how to prove it.

\subsection{Application: Metric TSP} \label{sec:main_app_TSP}

We first recall an important probabilistic bound from \cite{KKO21} used in the analysis of their metric TSP approximation algorithm (which is a slight modification of the max entropy algorithm from \cite{OGSS11} first studied by \cite{AGMOGS10}). In what follows, we let $A_S := \sum_{i \in S} A_i$ and $\kappa_S := \sum_{i \in S} \kappa_i$. See \cref{sec:tech_overview} for any undefined notation.

\begin{theorem}[Prop. 5.1 of \cite{KKO21}] \label{thm:doubly_exp}
    Let $\mu$ be a strongly Rayleigh distribution on $[m]$, let $A_1,\ldots,A_n$ be random variables counting the number of elements contained in disjoint subsets of $[m]$, and fix $\bm\kappa \in \Z^n$ with non-negative entries. Suppose for all $S \subseteq [n]$ we have
    \[
        \PP{\mu}{A_S \geq \kappa_S} \geq \epsilon \qquad \text{and} \qquad \PP{\mu}{A_S \leq \kappa_S} \geq \epsilon.
    \]
    Then we have
    \[
        \PP{\mu}{A_1 = \kappa_1, \ldots, A_n = \kappa_n} \geq f(\epsilon) \cdot \PP{\mu}{A_{[n]} = \kappa_{[n]}},
    \]
    where $f(\epsilon) \geq \epsilon^{2^n} \prod_{i=2}^n \frac{1}{\max\{\kappa_i, \kappa_{[i-1]}\} + 1}$.
\end{theorem}

In \cite{KKO21}, the authors note two things about this bound. First, they note that to apply the bound it is sufficient to have 
\begin{equation}\label{eq:expectation}
    \big|\EE{\mu}{A_S} - \kappa_S\big| < 1 \qquad \forall S \subseteq [n],
\end{equation}
since this implies a lower bound on $\PP{\mu}{A_S = \kappa_S}$ for all $S \subseteq [n]$ for strongly Rayleigh distributions. Second, they note that the bound on $f(\epsilon)$ is doubly exponential in $n$, but they expect the true dependency to only be simply exponential. They leave it as an open problem to determine a tight lower bound on $f(\epsilon)$.

In this paper, we further improve the metric TSP approximation factor by resolving a weak version of this open problem: we give a simply exponential lower bound which depends tightly on $\epsilon$, under the stronger condition of \eqref{eq:expectation}. Concretely, we prove the following. 

\begin{theorem}[Improved probability lower bound] \label{thm:main_imp_prob_bound}
    Let $\mu$ be a strongly Rayleigh distribution on $[m]$, let $A_1,\ldots,A_n$ be random variables counting the number of elements contained in disjoint subsets of $[m]$, and fix $\bm\kappa \in \Z^n$ with non-negative entries. Suppose for all $S \subseteq [n]$ we have
    \[
        \big|\EE{\mu}{A_S} - \kappa_S\big| \leq 1-\epsilon.
    \]
    Then we have
    \[
        \PP{\mu}{A_1 = \kappa_1, \dots, A_n = \kappa_n} \geq \epsilon^n \prod_{\kappa_i > 0} \frac{1}{e\sqrt{\kappa_i}}.
    \]
    The dependence on $\epsilon$ is tight for any fixed $\bm\kappa$ with strictly positive entries.
\end{theorem}


In fact, we prove stronger versions of this result which more directly depend on the specific values of $(\EE{\mu}{A_S} - \kappa_S)$ for all $S \subseteq [n]$; see \cref{thm:mainprobbound} and \cref{cor:mainprobbound_lite}. These results are analogous to our main coefficient bound \cref{cor:main_nh_coeff} because the coefficients and the gradient of probability generating polynomials can be interpreted as the probabilities and the expectation of the associated distribution. That said, our stronger probabilistic results require a more delicate analysis of the expectations (gradient) beyond what is required for \cref{cor:main_nh_coeff}. In particular, note that the conditions on the expectations in \cref{thm:main_imp_prob_bound} are more general than a bound on the $\ell^1$ norm of $(\EE{\mu}{A_i} - \kappa_i)_{i=1}^n$. See \cref{sec:proof_app_TSP} for further details.

Using \cref{thm:main_imp_prob_bound}, we improve the metric TSP approximation factor for the algorithm given in \cite{KKO21}.

\begin{theorem}\label{thm:improvedTSPconstant}
    There exists a randomized algorithm for metric TSP with approximation factor $\frac{3}{2} - \epsilon$ for some $\epsilon > 10^{-34}$. 
\end{theorem}

This is about a $100$ times improvement over the result of \cite{KKO22}. Thus our improvement in terms of the approximation factor itself may be smaller than anticipated, given that we were able to improve the probability bound in \cref{thm:main_imp_prob_bound} from doubly exponential to simply exponential. The reason for this is that while \cref{thm:doubly_exp} was useful in \cite{KKO21} to quickly determine which events occurred with constant probability (and indeed provided a single unifying explanation for why one should expect many of their probabilistic bounds to hold), it gave such small guarantees that \cite{KKO21} resorted to ad hoc arguments instead to give their final probabilistic bounds. 

We show that \cref{thm:main_imp_prob_bound} alone can be used to give bounds that are comparable to the ad hod methods of \cite{KKO21} (and, in several important cases, much better) whenever the bounds came purely from information on the expectations as in \eqref{eq:expectation}. Thus, we believe our main contribution to work on metric TSP is a version of \cref{thm:doubly_exp} that is ``reasonable" to use, allowing one to show a similar approximation factor but with a more streamlined proof. 


Unfortunately, not all of the bounds in \cite{KKO21} follow from expectation information, and two of them become bottlenecks for improving the approximation factor after applying \cref{thm:main_imp_prob_bound} to the other statements. Thus, to demonstrate \cref{thm:improvedTSPconstant} we need to sharpen these bounds using other techniques. 
For one of these lemmas we show that the existing proof in \cite{KKO21} was far from tight, and in the other we refine their proof. In particular, using  \cref{thm:main_imp_prob_bound}, we show we can reduce this second lemma to a special case that is possible to analyze more carefully.  See \cref{sec:proof_app_TSP} for further details.

\section{Technical Overview} \label{sec:tech_overview}

In this section we discuss the proof strategy of
our main capacity and coefficient bounds, \cref{thm:main_nh} and \cref{cor:main_nh_coeff}, and their stronger forms.

\paragraph{Notation.}

Given a vector $\bm{z} \in \R^{E}$ and a subset $S$ of $E$, let $\bm{z}^S := \prod_{e \in S} z_e$. Let $\mu: \{0,1\}^{E} \to \R$ be a probability distribution over subsets of $E$. The generating polynomial $g_\mu \in  \R_{\geq 0}[\{z_{e}\}_{e\in E}]$ of $\mu$ is defined as
$$ g_\mu(z):=\sum_{S \subseteq E} \mu(S) \cdot \bm{z}^S.$$
The distribution $\mu$ is \textbf{strongly Rayleigh} if $g_\mu$ is real stable, where a polynomial $p \in \mathbb{R}[z_1,\dots,z_n]$ is \textbf{real stable} if $p(\bm{z}) \neq 0$ whenever $\Im(z_i) > 0$ for all $i \in [n]$ (i.e., when all inputs are in the complex upper half-plane). See \cite{BBL09} for much more on strongly Rayleigh measures. Further, given a polynomial $p \in \R_{\geq 0}[x_1,\ldots,x_n]$ and $\bm\kappa \in \Z_{\geq 0}^n$, the \textbf{capacity} of $p$ is defined as
\[
    \cpc_{\bm\kappa}(p) := \inf_{\bm{x} > 0} \frac{p(\bm{x})}{\bm{x}^{\bm\kappa}}.
\]
Finally, we let $p_{\bm\kappa}$ denote the coefficient of $\bm{x}^{\bm\kappa}$ in $p$.

\paragraph{Conceptual strategy.}

We first give an overarching view of the strategy used to prove our main results, as well as the key similarities and differences compared to that of \cite{GL21}. The general idea for proving our bounds is to find a simple and sparse underlying structure for the worst-case inputs. The space of all real stable polynomials can be complicated, but we show that the worst-case polynomials for our bounds are far simpler: they are ``sparse'' products of affine linear forms. More concretely, we reduce the space of input polynomials (and the corresponding combinatorial structures) as follows:
\begin{center}
\begin{tabular}{ccccc}
    real stable polynomials & $\implies$ & products of affine linears & $\implies$ & sparse products of affine linears \\
    matroids & $\implies$ & matrices & $\implies$ & forests
\end{tabular}
\end{center}
The first reduction step uses the idea of productization which was the key idea from \cite{GL21}. This allows for one to utilize the matrix structure inherent to products of affine linear forms.

The second reduction step is then new to this paper. We first show that we may restrict to the extreme points of the set of matrices corresponding to products of affine linear forms, and then we show that these extreme matrices are supported on the edges of forests. This implies a significant decrease in density of the matrices: general graphs can have quadratically many edges, whereas forests can only have linearly many. This allows for an intricate but clean induction on the leaf vertices of these forests, which yields the strongest bounds of this paper. Additionally, it is this step that allows for bounds which do not depend on the total degree of the polynomial, and this was a crucial barrier to applying the bounds of \cite{GL21} to metric TSP.

\subsection{Conceptual Strategy, in More Detail} \label{sec:strat_outline}

We now go through the steps of the conceptual strategy described above in more detail. Let us first restate our main capacity and coefficient bounds.

\begin{theorem}[= \cref{thm:main_nh} and \cref{cor:main_nh_coeff}] \label{thm:tech_main_bounds}
    Let $p \in \R_{\geq 0}[x_1,\ldots,x_n]$ be a real stable polynomial in $n$ variables, and fix any $\bm\kappa \in \Z^n$ with non-negative entries.
    If $p(\bm{1}) = 1$ and $\|\bm\kappa - \nabla p(\bm{1})\|_1 < 1$, then
    \[
        \cpc_{\bm\kappa}(p) \geq (1 - \|\bm\kappa - \nabla p(\bm{1})\|_1)^n
    \]
    and
    \[
        p_{\bm\kappa} \geq \left(\prod_{i=1}^n \frac{\kappa_i^{\kappa_i} e^{-\kappa_i}}{\kappa_i!}\right) \left(1 - \|\bm\kappa - \nabla p(\bm{1})\|_1\right)^n,
    \]
    where $\cpc_{\bm\kappa}(p) := \inf_{\bm{x} > 0} \frac{p(\bm{x})}{\bm{x}^{\bm\kappa}}$ is the \textbf{capacity} of $p$ and $p_{\bm\kappa}$ is the coefficient of $\bm{x}^{\bm\kappa}$ in $p$. The dependence on $(1 - \|\bm\kappa - \nabla p(\bm{1})\|_1)$ in these bounds is tight for any fixed $\bm\kappa$ with strictly positive entries.
\end{theorem}

Analogous bounds required for the metric TSP application then follow from interpreting desired quantities as the coefficients and gradient of certain real stable polynomials. Specifically, the strongly Rayleigh probabilities we wish to lower bound are the coefficients of the corresponding real stable generating polynomial, and the expectations of the associated random variables are given by the gradient of that polynomial. We leave further details to \cref{sec:proof_app_TSP}.

We now discuss the proof of \cref{thm:tech_main_bounds}. First note that the coefficient bound follows from the capacity bound. This immediately follows from Corollary 3.6 of \cite{Gur09}, which implies
\begin{equation} \label{eq:G-bound}
    p_{\bm\kappa} \geq \left(\prod_{i=1}^n \frac{\kappa_i^{\kappa_i} e^{-\kappa_i}}{\kappa_i!}\right) \cpc_{\bm\kappa}(p).
\end{equation}
Thus what remains to be proven is the capacity bound
\[
    \cpc_{\bm\kappa}(p) \geq \left(1 - \|\bm\kappa - \nabla p(\bm{1})\|_1\right)^n,
\]
which is precisely the bound of \cref{thm:main_nh}, as well as its tightness, which follows from considering a particular example of $p$ (see \cref{lem:tightness}).

The remainder of the proof then has four main steps. We also note here that these proof steps actually imply stronger bounds than \cref{thm:main_nh}, see \cref{cor:strong_bounds_stable} for the formal statement. These stronger bounds are required for the metric TSP application.

\paragraph{Step 1: Reduce to products of affine linear forms via productization.}

We first generalize the productization technique of \cite{GL21} to non-homogeneous real stable polynomials. The upshot of this technique is that it implies it is sufficient to prove \cref{thm:main_nh} for products of affine linear forms with non-negative coefficients (see \cref{cor:prod-cap-bound}). Such polynomials correspond to $d \times (n+1)$ $\R_{\geq 0}$-valued matrices with row sums $\bm{1}$ and column sums $\bm\alpha$ equal the entries of the gradient of the polynomial, via
\[
    \phi: A \mapsto \prod_{i=1}^d \left(a_{i,n+1} + \sum_{j=1}^n a_{i,j} x_j\right).
\]
This gives far more structure to work with, beyond that of real stable polynomials in general. This part is a straightforward generalization of the analogous result of \cite{GL21}.

\paragraph{Step 2: Reduce to extreme points.}

The set of $\R_{\geq 0}$-valued matrices with row sums $\bm{1}$ and column sums $\bm\alpha$ forms a convex polytope $P_{\bm\alpha}^d$, and thus the polynomials we now must consider correspond to the points of this polytope via the map $\phi$ defined above. Inspired by a result of Barvinok (see \cref{lem:log-concave_cap}), we next show that the function
\[
    A \mapsto \cpc_\kappa(\phi(A))
\]
is log-concave on the above described polytope. Since we want to minimize the capacity, this implies we may further restrict to the polynomials associated to the extreme points of the polytope.

\paragraph{Step 3: Extreme points correspond to bipartitioned forests.}

Any $\R_{\geq 0}$-valued matrix $A$ can be interpreted as the weighted bipartite adjacency matrix of a bipartite graph, where the left vertices correspond to the rows of $A$ and the right vertices correspond to the columns of $A$. A matrix $A \in P_{\bm\alpha}^d$ being extreme implies the associated bipartite graph has no cycles. This implies the associated bipartite graph is a forest (see \cref{lem:forest_extreme}). The sparsity properties of such matrices implies a simple structure for the associated polynomials, which is particularly amenable to an intricate but clean induction.\footnote{Note that it is already mentioned in \cite{GL21} that the supports of the extreme points correspond to forests, but the application of this observation in \cite{GL21} is somewhat ``naive'' and kind of brute-force: it was mainly used to describe an (inefficient) algorithm to compute the capacity lower bound for products of linear forms, which is not related to main lower bound in this paper. Additionally, its use in \cite{GL21} is quite conceputally far from how it could actually be used to improve the TSP approximation factor.}

\paragraph{Step 4: Induction on leaf vertices of the bipartitioned forests.}

Leaf vertices in the forest corresponding to a given matrix $A \in P_{\bm\alpha}^d$ indicate rows or columns of the matrix $A$ which have exactly one non-zero entry. If a row of $A$ has exactly one non-zero entry, the induction proceeds in a straightforward fashion, by simply removing the corresponding row of $A$ and recalculating the column sums (see the $d \geq n+1$ case of the proof of \cref{thm:two_term_cap_bound}).

If a column of $A$ (say column $i$) has exactly one non-zero entry, then the induction is more complicated. We prove lemmas showing how much the capacity can change after applying the partial derivative $\partial_{x_i}$ (when $\kappa_i \geq \alpha_i$, see \cref{lem:bound_kappa_1}) or setting $x_i$ to 0 (when $\kappa_i < \alpha_i$, see \cref{lem:bound_kappa_0}). Since column $i$ has only one entry, applying $\partial_{x_i}$ corresponds to removing column $i$ and the row of $A$ which contains the non-zero entry, and setting $x_i$ to $0$ corresponds to removing column $i$. After renormalizing the row sums and recalculating the column sums, the proof again proceeds by induction. (See the proof of \cref{thm:two_term_cap_bound} to see the above arguments presented formally.) \cref{ex:delta_tight} and \cref{ex:epsilon_tight} show that the distinction between the $\kappa_i \geq \alpha_i$ and $\kappa_i < \alpha_i$ cases is not an artifact of the proof.

\paragraph{Some comments on tightness of the bounds.}

The main coefficient bound of \cref{cor:main_nh_coeff} is proven via two different bounds, as discussed at the beginning of this section. That is, one first bounds the coefficient in terms of the capacity \eqref{eq:G-bound} via Corollary 3.6 of \cite{Gur09}, and then one bounds the capacity via the steps outlined above. Thus while the capacity bound (\cref{thm:main_nh}) is tight for $\bm\kappa > \bm{0}$, the coefficient bound (\cref{cor:main_nh_coeff}) may not be. We note that the coefficient bound is likely close to tight in the case that $(1 - \|\bm\kappa - \nabla p(\bm{1})\|_1)$ is close to $1$, but it seems this tightness deteriorates as $(1 - \|\bm\kappa - \nabla p(\bm{1})\|_1)$ gets close to $0$. That said, the dependence on $(1 - \|\bm\kappa - \nabla p(\bm{1})\|_1)$ in \cref{thm:main_nh} and \cref{cor:main_nh_coeff} is tight for $\bm\kappa > \bm{0}$ by \cref{lem:tightness}, though there is still potential for improvement in the more refined capacity bounds; see \cref{sec:examples}.

That said, tight coefficient lower bounds in the univariate case can be achieved by directly applying Hoeffding's theorem (\cref{thm:hoeffding}), and these lower bounds resemble the coefficient lower bound of \cref{cor:main_nh_coeff}. Thus one can view \cref{cor:main_nh_coeff} as a step towards a multivariate generalization of Hoeffding's theorem. It is an interesting question whether or not the techniques used here can be extended to a full multivariate generalization of Hoeffding's theorem.

\subsection{Example: The Univariate Case} \label{sec:univ_case}

In this section, we demonstrate the proof of \cref{thm:main_nh} in the univariate case. This will serve as a sort of proof of concept for the more general proof.

Here we consider real stable non-homogeneous polynomials $p \in \R_{\geq 0}[x_1]$ such that $p(1) = 1$ and $\nabla p(1) = \alpha_1$, and we define $\epsilon := 1 - |\alpha_1 - \kappa_1| > 0$. By Step 1 above, we may assume that $p$ is of the form
\[
    p(x_1) = \prod_{i=1}^d (a_{i,1} x_1 + a_{1,2}),
\]
where $A$ is an $\R_{\geq 0}$-valued $d \times 2$ matrix with row sums $1$ and column sums $(\alpha_1,d-\alpha_1)$. By Steps 2-3, we may further assume $A$ is the weighted bipartite adjacency matrix of a forest. If every row of $A$ contains exactly one non-zero entry then $p(x_1) = x_1^{k_1}$, and the result is trivial in this case. Otherwise, $d-1$ rows of $M$ have exactly one non-zero entry (see \cref{lem:left_leaves}). Thus for some $k \leq d-1$ we have
\[
    p(x_1) = x_1^k (a x_1 + b),
\]
where $a+b = 1$ and $a + k = \alpha_1$. Since $\kappa_1 \in \Z_{\geq 0}$ and
\[
    \epsilon = 1 - |\alpha_1 - \kappa_1| = 1 - |a + k - \kappa_1|,
\]
we have that $k$ is equal to either $\kappa_1$ or $\kappa_1 - 1$. If $k = \kappa_1$, then
\[
    \cpc_{\kappa_1}(p) = \inf_{x_1 > 0} \frac{x_1^k(ax_1+b)}{x_1^k} = b \qquad \text{and} \qquad \epsilon = 1 - |a + k - \kappa_1| = 1-a = b.
\]
If $k = \kappa_1 - 1$, then
\[
    \cpc_{\kappa_1}(p) = \inf_{x_1 > 0} \frac{x_1^k(ax_1+b)}{x_1^{k+1}} = a \qquad \text{and} \qquad \epsilon = 1 - |a + k - \kappa_1| = 1 - (1-a) = a.
\]
Therefore in both cases we have
\[
    \cpc_{\kappa_1}(p) = \epsilon = (1 - |\alpha_1 - \kappa_1|),
\]
which proves \cref{thm:main_nh} in the univariate case and demonstrates the tight dependence on $(1 - \|\bm\kappa - \nabla p(\bm{1})\|_1)$ in this case.



\section{Proof of Application: Metric TSP} \label{sec:proof_app_TSP}

The following is the strongest probability lower bound, which we will use for the metric TSP application. In what follows, we let $A_S := \sum_{i \in S} A_i$ and $\kappa_S := \sum_{i \in S} \kappa_i$. Note that if
\[
    \big|\EE{\mu}{A_S} - \kappa_S\big| \leq 1-\epsilon,
\]
then \cref{thm:main_imp_prob_bound} immediately follows from \cref{thm:mainprobbound}, \cref{lem:tightness}, and a standard computation.


\begin{theorem}[Strongest form of the probability bound]\label{thm:mainprobbound}
    Let $\mu$ be a strongly Rayleigh distribution on $[m]$, let $A_1,\ldots,A_n$ be random variables counting the number of elements contained in some associated disjoint subsets of $[m]$, and fix $\bm\kappa \in \Z_{\geq 0}^n$. Suppose for all $S \subseteq [n]$ we have $|\EE{\mu}{A_S} - \kappa_S| < 1$. Define $\bm\epsilon,\bm\delta \in \R_{> 0}^n$ via
    \[
        \delta_k := 1 + \min_{S \in \binom{[n]}{k}} \left(\EE{\mu}{A_S} - \kappa_S\right) \qquad \text{and} \qquad \epsilon_k := 1 - \max_{S \in \binom{[n]}{k}} \left(\EE{\mu}{A_S} - \kappa_S\right).
    \]
    Then we have
    \[
        \PP{\mu}{A_1 = \kappa_1, \ldots, A_n = \kappa_n} \geq \prod_{i=1}^n \frac{\kappa_i^{\kappa_i} e^{-\kappa_i}}{\kappa_i!} \cdot \min_{0 \leq \ell \leq n} \prod_{k=1}^\ell \epsilon_k \prod_{k=1}^{n-\ell} \delta_k.
    \]
\end{theorem}
\begin{proof}
    Let $q$ be the probability generating polynomial of $\mu$, and let $p$ be the polynomial obtained by setting the variables of $q$ associated to $A_i$ to $x_i$ for all $i \in [n]$, and setting all other variables to equal $1$. Since $\mu$ is strongly Rayleigh, $p$ is real stable. Further, $p(\bm{1}) = 1$, $\nabla p(\bm{1}) = (\EE{\mu}{A_i})_{i=1}^n$, and $\PP{\mu}{A_1 = \kappa_1, \ldots, A_n = \kappa_n}$ is the $\bm{x}^{\bm\kappa}$ coefficient of $p$. Thus Gurvits' capacity inequality \eqref{eq:G-bound} and \cref{cor:strong_bounds_stable} imply the desired result.
\end{proof}

We also give a slightly weaker bound which is a bit easier to use in practice. Note that \cref{thm:main_imp_prob_bound} also follows from \cref{cor:mainprobbound_lite}.

\begin{corollary} \label{cor:mainprobbound_lite}
    Let $\mu$ be a strongly Rayleigh distribution on $[m]$, let $A_1,\ldots,A_n$ be random variables counting the number of elements contained in some associated disjoint subsets of $[m]$, and fix $\bm\kappa \in \Z_{\geq 0}^n$. Suppose for all $S \subseteq [n]$ we have $|\EE{\mu}{A_S} - \kappa_S| < 1$. Define $\bm\epsilon \in \R_{>0}^n$ via
    \[
        \epsilon_k := 1 - \max_{\substack{S \subseteq [n] \\ |S| \leq k}} \big|\EE{\mu}{A_S} - \kappa_S\big|
    \]
    Then we have
    \[
        \PP{\mu}{A_1 = \kappa_1, \ldots, A_n = \kappa_n} \geq \prod_{i=1}^n \frac{\kappa_i^{\kappa_i} e^{-\kappa_i}}{\kappa_i!} \cdot \prod_{k=1}^n \epsilon_k.
    \]
\end{corollary}
\begin{proof}
    Follows from \cref{thm:mainprobbound}; see the proof of \cref{cor:one_term_cap_bound} for more details.
\end{proof}

The remainder of this section is devoted to demonstrating how one can use the above results to improve the approximation factor for metric TSP.



%

\subsection{A Simple Application}

We recall some definitions from \cite{KKO21}. There, we have a graph $G=(V,E)$ and a strongly Rayleigh (SR) distribution $\mu:2^E \to \R_{\ge 0}$ supported on spanning trees of $G$. We let $x_e = \PP{T \sim \mu}{e \in T}$ and for  a set of edges $F$ let $x(F) = \sum_{e \in F} x_e$.  Furthermore we let $\delta(S) = \{e=\{u,v\} \mid |e \cap S|=1\}$. The guarantee on $x$ is that $x \in \subLP$,\footnote{Technically, a spanning tree plus an edge is sampled, as otherwise one cannot exactly have $x \in \subLP$, but we ignore that here.} where
\begin{equation}
\begin{aligned}
\subLP := \begin{cases}
x(\delta(S))  \ge 2 & \forall \, S\subsetneq V\\
x(\delta(v)) = 2 & \forall v \in V\\
x_{\{u,v\}}\geq  0  &  \forall \, u,v \in V
\end{cases}
\end{aligned}
\label{eq:tsplp}
\end{equation}

In the algorithm they analyze, one first samples a spanning tree $T$ from $\mu$ and then add the minimum cost matching on the odd vertices of $T$. Their work involves analyzing the expected cost of this matching over the randomness of the sampled tree $T$. Unsurprisingly, the \textit{parity} of vertices is therefore very important, as it determines which vertices are involved in the matching. 

In this section, to give some sense of the utility of \cref{thm:mainprobbound}, we show an application in a simplified setting. Namely, we show that for any two vertices $u,v$, except in the special case that $x_{\{u,v\}} \approx \frac{1}{2}$, we have $\PP{T \sim \mu}{|\delta(u) \cap T|=|\delta(v) \cap T| = 2} \ge \Omega(1)$, i.e. for any two vertices that do not share an edge of value $\frac{1}{2}$ there is a constant probability that they have even parity simultaneously. This is helpful because (under some conditions on the point $x \in \subLP$) the event that $u,v$ are even simultaneously indicates one can strictly decrease the cost of the matching proportional to the cost of the edge $e$.

We now prove $\PP{T \sim \mu}{|\delta(u) \cap T|=|\delta(v) \cap T| = 2} \ge \Omega(1)$ whenever $x_{\{u,v\}} \not\approx \frac{1}{2}$. To do this, we split into two cases: when $x_{\{u,v\}} \geq \frac{1}{2} + \epsilon$ and when $x_{\{u,v\}} \leq \frac{1}{2} - \epsilon$.

\begin{lemma}
Let $u,v$ be two vertices such that $x_{\{u,v\}} \ge \frac{1}{2}+\epsilon$ for some $\epsilon > 0$. Then, 
$$\PP{T \sim \mu}{|\delta(u) \cap T|=|\delta(v) \cap T|=2} \ge \frac{\epsilon}{2e^3}.$$
\end{lemma}
\begin{proof}
    Let $e=(u,v)$. For $T \sim \mu$, let $A_1=\mathbb{I}[e \in T]$, $A_2 = |(\delta(u) \smallsetminus \{e\}) \cap T|$, $A_3 = |(\delta(v) \smallsetminus \{e\}) \cap T|$. We are now interested in the event $A_i = \kappa_i, \forall i$ for the vector $\kappa=(1,1,1)$ as this implies $|\delta(u) \cap T|=|\delta(v) \cap T|=2$. We have $\E[A_1] = x_e$, $\E[A_2]=\E[A_3]=2-x_e$. Therefore, to apply \cref{thm:mainprobbound} we can set:
    \begin{align*}
        \delta_1 = x_e, \delta_2 = 1, \delta_3 = 2-x_e, \qquad \epsilon_1 = x_e, \epsilon_2 = 2x_e-1, \epsilon_3 = x_e 
    \end{align*}
In this case the worst case is using all of the $\epsilon$ terms in the bound, giving $e^{-3}x_e^2(2x_e-1) \ge \frac{\epsilon}{2e^3}$ as desired.
\end{proof}

\begin{lemma}
Let $u,v$ be two vertices such that $x_{\{u,v\}} \le \frac{1}{2}-\epsilon$ for some $\epsilon > 0$. Then, 
$$\PP{T \sim \mu}{|\delta(u) \cap T|=|\delta(v) \cap T|=2} \ge \frac{2\epsilon}{e^4}.$$
\end{lemma}
\begin{proof}
    As above, $e=(u,v)$, and for $T \sim \mu$, let $A_1=\mathbb{I}[e \in T]$, $A_2 = |(\delta(u) \smallsetminus \{e\}) \cap T|$, $A_3 = |(\delta(v) \smallsetminus \{e\}) \cap T|$. We are now interested in the event $A_i = \kappa_i, \forall i$ for the vector $\kappa=(0,2,2)$ as this implies $|\delta(u) \cap T|=|\delta(v) \cap T|=2$. We have $\E[A_1] = x_e$, $\E[A_2]=\E[A_3]=2-x_e$. Therefore, to apply \cref{thm:mainprobbound} we can set:
    \begin{align*}
        \delta_1 = 1-x_e, \delta_2 = 1-2x_e, \delta_3 = 1-x_e, \qquad \epsilon_1 = 1-x_e, \epsilon_2 = 1, \epsilon_3 = 1+x_e 
    \end{align*}
In this case the worst case is using all of the $\delta$ terms in the bound, giving $4e^{-4}(1-x_e)^2(1-2x_e) \ge \frac{2\epsilon}{e^4}$.
\end{proof}

We leave further background about TSP and the proofs of our improved probabilistic statements to \cref{sec:app}, as understanding their importance requires some knowledge of the (highly technical) proof in \cite{KKO21}. However, in the next section we summarize the new probabilistic bounds we get and their consequences. 

\subsection{Summary of Probabilistic Bounds and New Approximation Factor} \label{sec:summary_tsp_bounds}

The current bound on the performance of the max entropy algorithm is $\frac{3}{2}-4.11 \cdot 10^{-36}$. This is primarily governed by a constant $p$ that is determined by the minimum probability over a number of events. In \cite{KKO21}, $p$ was equal to $2 \cdot 10^{-10}$, and these events are described by the following statements in \cite{KKO21}:
\begin{enumerate}
\item Corollary 5.9, which gives a bound of $1.5 \cdot 10^{-9}$. We do not modify this bound, although we note that Lemma 5.7 can be slightly improved which would lead to a small improvement here. 
\item Lemma 5.21, which gives a bound of $0.005 \epsilon_{1/2}^2 = 2 \cdot 10^{-10}$. We improve this to $0.039 \epsilon_{1/2}^2 = 1.56 \cdot 10^{-9}$. 
\item Lemma 5.22, which gives a bound of $0.006 \epsilon_{1/2}^2 = 2.4 \cdot 10^{-10}$. We improve this to $0.038 \epsilon_{1/2}^2 = 1.52 \cdot 10^{-9}$. 
\item Lemma 5.23, which gives a bound of $0.005 \cdot \epsilon_{1/2} = 2 \cdot 10^{-10}$. We observe here that arguments already in \cite{KKO21} can be used to give $0.0498 \epsilon_{1/2}^2 \ge 1.9 \cdot 10^{-9}$. 
\item Lemma 5.24, which gives a bound of $0.02\epsilon_{1/2}^2 = 8 \cdot 10^{-10}$. We improve this to $0.0485\epsilon_{1/2}^2 \ge  1.9 \cdot 10^{-9}$.
\item Lemma 5.27, which gives a bound of 0.01. This lemma actually uses that the threshold $p$ is \textit{small}, and therefore this bound decreases slightly upon raising $p$. However, as it is quite far from being the bottleneck in these bounds, we omit the proof that the probability remains above $1.5 \cdot 10^{-9}$.
\end{enumerate}
We prove these bounds formally using our new strongly Rayleigh probability bounds in \cref{sec:cap-prob-improvements} (which directly uses our new results) and \cref{app:furtherProb} (which mixes in more ad-hoc methods).

Therefore, we may increase $p$ to the minimum of all these probabilities, $1.5 \cdot 10^{-9}$. In \cref{sec:app}, we observe that using statements from \cite{KKO21,KKO22}, the following holds:
\begin{restatable}{lemma}{TSPblackbox}
    Let $p$ be a lower bound on the probabilities guaranteed by (1) - (6) for $\epsilon_{1/2} \le 0.0002$, and suppose $p \le 10^{-4}$. Then given $x \in \subLP$, the max entropy algorithm returns a solution of expected cost at most $(\frac{3}{2}-9.7p^2 \cdot 10^{-17}) \cdot c(x)$. 
\end{restatable}
 As we improve the bounds on $p$ to $1.5 \cdot 10^{-9}$, an immediate corollary is the following:
\begin{corollary}
    The max entropy algorithm is a $\frac{3}{2}-2.18 \cdot 10^{-34}$ approximation algorithm for metric TSP.
\end{corollary}
Using \cite{KKO23} this guarantee can be made deterministic, as we do not require any modifications to the algorithm.

In \cref{sec:app}, we also observe a lower bound on Lemma 5.21 for strongly Rayleigh distributions of $\Omega(\epsilon_{1/2}^2)$. The fact that $\epsilon_{1/2} \le 0.0002$ is used in many places in \cite{KKO21} and thus decreasing it may require more effort. Thus without modifying other parts of the argument, it may not be possible to improve the bound below $1.5 \cdot 10^{-31}$. 

In the rest of the body of the paper we prove our main capacity bound.

\section{Proofs of the Main Capacity Bounds} \label{sec:main_proofs}

In this section we prove the strongest forms of the main capacity bounds, which give \cref{thm:main_nh} and \cref{cor:main_nh_coeff} as corollaries. For this section, we utilize the following notation.

\begin{definition}
    For $n,d \in \N$ and $\bm\alpha \in \R_{\geq 0}^n$, we define the following:
    \begin{enumerate}
        \item $\NHMat_n^d(\bm\alpha)$ is the set of all $\R_{\geq 0}$-valued $d \times (n+1)$ matrices with row sums all equal to 1 and column sums equal to $\alpha_1,\ldots,\alpha_n,d-\|\bm\alpha\|_1$.
        \item $\NHProd_n^d(\bm\alpha)$ is the set of all polynomials of the form
        \[
            p(x) = \prod_{i=1}^d \left(a_{i,n+1} + \sum_{j=1}^n a_{i,j} x_j\right),
        \]
        where $A \in \NHMat_n^d(\bm\alpha)$. In this case, we call $p$ the \emph{polynomial associated to $A$}. Note that $p(\bm{1}) = 1$ and $\nabla p(\bm{1}) = \bm\alpha$ for all such polynomials.
        \item $\NHStab_n^d(\bm\alpha)$ is the set of all real stable polynomials in $\R_{\geq 0}[x_1,\ldots,x_n]$ of degree at most $d$ for which $p(\bm{1}) = 1$ and $\nabla p(\bm{1}) = \bm\alpha$. (Recall that a polynomial is stable if it is never zero when all inputs are in the open complex upper half-plane.)
    \end{enumerate}
\end{definition}

We also define the following for $n \in \N$, $\bm\alpha \in \R_{\geq 0}^n$, and $\bm\kappa \in \Z_{\geq 0}^n$:
\[
    L^{\NHProd}_n(\bm\alpha;\bm\kappa) := \min_{d \in \N} \min_{p \in \NHProd_n^d(\bm\alpha)} \cpc_{\bm\kappa}(p).
\]
We now follow the steps of the proof given in \cref{sec:tech_overview}.

\subsection{Productization for Non-homogeneous Stable Polynomials}

We first show how we can reduce the problem of bounding the capacity to products of affine linear forms. We recall the main productization result from \cite{GL21}, which gives the non-homogeneous productization result as an immediate corollary.

\begin{theorem}[Thm. 6.2, \cite{GL21}] \label{thm:hstab-prod}
    Fix $n,d \in \N$, $\bm{u},\bm\alpha \in \R_{\geq 0}^n$, and $p \in \R_{\geq 0}[x_1,\ldots,x_n]$ of homogeneous degree $d$, such that $p(\bm{1}) = 1$ and $\nabla p(\bm{1}) = \bm\alpha$. There exists an $\R_{\geq 0}$-valued $d \times n$ matrix $A$ such that the rows sums of $A$ are all equal to $\bm{1}$, the column sums of $A$ are given by $\bm\alpha$, and $p(\bm{u}) = \prod_{i=1}^d (A\bm{u})_i$.
\end{theorem}

\begin{corollary} \label{cor:nhstab-prod}
    Fix $n,d \in \N$, $\bm{u},\bm\alpha \in \R_{\geq 0}^n$, and $p \in \NHStab_n^d(\bm\alpha)$. There exists $f \in \NHProd_n^d(\bm\alpha)$ such that $p(\bm{u}) = f(\bm{u})$.
\end{corollary}
\begin{proof}
    Let $q(\bm{x}) = x_{n+1}^d \cdot p\left(\frac{x_1}{x_{n+1}}, \ldots, \frac{x_n}{x_{n+1}}\right)$ be the homogenization of $p$, and define $\bm\beta := \nabla q(\bm{1})$. So $q \in \R_{\geq 0}[x_1,\ldots,x_{n+1}]$ of homogeneous degree $d$ such that $q(\bm{1}) = 1$ and $\nabla q(\bm{1}) = \bm\beta = (\alpha_1,\ldots,\alpha_n,d-\|\bm\alpha\|_1)$. Define $u_{n+1} := 1$, apply \cref{thm:hstab-prod} to $q$ and $\bm{u}$, and dehomogenize to obtain the desired result. 
\end{proof}

We now use this result to reduce the problem of bounding the capacity to products of affine linear forms.

\begin{corollary} \label{cor:prod-cap-bound}
    For $p \in \NHStab_n^d(\bm\alpha)$, we have
    \[
        \cpc_\kappa(p) \geq L^{\NHProd}_n(\bm\alpha;\bm\kappa).
    \]
\end{corollary}
\begin{proof}
    For any $\bm{x} \in \R_{\geq 0}^n$, let $f \in \NHProd_n^d(\bm\alpha)$ be such that $p(\bm{x}) = f(\bm{x})$ according to \cref{cor:nhstab-prod}. With this, we have
    \[
        \cpc_{\bm\kappa}(p) = \inf_{\bm{x} > 0} \frac{p(\bm{x})}{\bm{x}^{\bm\kappa}} \geq \inf_{\bm{x} > 0} \min_{d \in \N} \min_{f \in \NHProd_n^d(\bm\alpha)} \frac{f(\bm{x})}{\bm{x}^{\bm\kappa}} = L_n^{\NHProd}(\bm\alpha;\bm\kappa).
    \]
\end{proof}

\subsection{The Extreme Points of $\NHMat_n^d(\alpha)$}

The next result implies we can reduce to the extreme points of $\NHMat_n^d(\bm\alpha)$ to lower bound $L_n^{\NHProd}(\bm\alpha;\bm\kappa)$.

\begin{lemma}[See Thm. 3.1 of \cite{Barv08}] \label{lem:log-concave_cap}
    Given $\bm\kappa \in \Z_{\geq 0}^n$ and $\bm\alpha \in \R_{\geq 0}^n$, let $\phi: \NHMat_n^d(\bm\alpha) \to \R_{\geq 0}$ be the function which maps $M$ to $\cpc_{\bm\kappa}(p)$ where $p$ is the polynomial associated to $M$. Then $\phi$ is log-concave on $\NHMat_n^d(\bm\alpha)$.
\end{lemma}
\begin{proof}
    See \cref{lem:log-concave_cap_proof}.
\end{proof}

We next describe the extreme points of $\NHMat_n^d(\bm\alpha)$ via bipartitioned forests.

\begin{lemma} \label{lem:forest_extreme}
    Any extreme point of $\NHMat_n^d(\bm\alpha)$ has support given by a bipartite forest on $d$ left vertices and $n+1$ right vertices.
\end{lemma}
\begin{proof}
    Let $M$ be an extreme point of $\NHMat_n^d(\bm\alpha)$, and suppose its bipartite support graph $G$ does not give a forest. Then $G$ must contain an even simple cycle. Group the edges of this cycle into two groups such that the odd edges make up one group, and the even edges make up the other (with any starting point). Add $\epsilon > 0$ to all matrix entries corresponding to even edges and subtract $\epsilon$ to all matrix entries corresponding to odd edges, to construct $M_+ \in \NHMat_n^d(\bm\alpha)$. Do the same thing, but reverse the signs, to construct $M_- \in \NHMat_n^d(\bm\alpha)$. Thus $M = \frac{M_+ + M_-}{2}$, contradicting the fact that $M$ is an extreme point.
\end{proof}

In what follows, we will also need the following basic graph theoretic result.

\begin{lemma} \label{lem:left_leaves}
    Let $G$ be a bipartite forest on $m$ left vertices and $n$ right vertices such that $G$ has no vertices of degree 0. Then $G$ has at least $m-n+1$ left leaves.
\end{lemma}
\begin{proof}
    See \cref{lem:left_leaves_proof}.
\end{proof}

\subsection{Capacity Bounds via Induction}

Here we complete the proof of \cref{thm:main_nh}. We first prove a simple lemma, which bears some resemblance to the probabilistic union bound.

\begin{lemma} \label{lem:prod_sum}
    Given $\bm{c} \in \R_{\geq 0}^d$ such that $c_i < 1$ for all $i \in [d]$, we have $1 - \sum_{i=1}^d c_i \leq \prod_{i=1}^d (1-c_i)$.
\end{lemma}
\begin{proof}
    See \cref{lem:prod_sum_proof}.
\end{proof}

The next lemma handles the case from Step 4 in \cref{sec:tech_overview} of setting some variable equal to $0$. To see how these next two lemmas actually are actually used, see \cref{thm:two_term_cap_bound} below.

\begin{lemma} \label{lem:bound_kappa_0}
    For $n \geq 1$, let $p \in \NHProd_n^d(\bm\alpha)$ be the polynomial associated to a $d \times (n+1)$ matrix $M$, and suppose $\bm\kappa \in \Z_{\geq 0}^n$ such that $\kappa_n = 0$ and $\alpha_n-\kappa_n \leq 1-\epsilon$ for some $\epsilon > 0$. Then there exists $q \in \NHProd_{n-1}^d(\bm\beta)$ such that
    \[
        \cpc_{\bm\kappa}(p) \geq \epsilon \cdot \cpc_{\bm\gamma}(q),
    \]
    where $\bm\gamma = (\kappa_1,\ldots,\kappa_{n-1}) \in \Z_{\geq 0}^{n-1}$ and $\bm\beta \in \R_{\geq 0}^{n-1}$ is such that for all $S \subseteq [n-1]$ we have
    \[
        \sum_{j \in S} (\alpha_j-\kappa_j) \leq \sum_{j \in S} (\beta_j-\gamma_j) \leq (\alpha_n-\kappa_n) + \sum_{j \in S} (\alpha_j-\kappa_j).
    \]
\end{lemma}
\begin{proof}
    See \cref{lem:bound_kappa_0_proof}.
\end{proof}

The next lemma handles the case from Step 4 in \cref{sec:tech_overview} of taking the partial derivative with respect to some variable.

\begin{lemma} \label{lem:bound_kappa_1}
    For $n \geq 1$, let $p \in \NHProd_n^d(\bm\alpha)$ be the polynomial associated to a $d \times (n+1)$ matrix $M$ such that column $n$ has exactly one non-zero entry, and suppose $\bm\kappa \in \Z_{\geq 0}^n$ is such that $\kappa_n = 1$ and $\kappa_n-\alpha_n \leq 1-\epsilon$ for some $\epsilon > 0$. Then there exists $q \in \NHProd_{n-1}^{d-1}(\bm\beta)$ such that
    \[
        \cpc_{\bm\kappa}(p) \geq \epsilon \cdot \cpc_{\bm\gamma}(q),
    \]
    where $\bm\gamma = (\kappa_1,\ldots,\kappa_{n-1}) \in \Z_{\geq 0}^{n-1}$ and $\bm\beta \in \R_{\geq 0}^{n-1}$ is such that for all $S \subseteq [n-1]$ we have
    \[
        \sum_{j \in S} (\kappa_j-\alpha_j) \leq \sum_{j \in S} (\gamma_j-\beta_j) \leq (\kappa_n-\alpha_n) + \sum_{j \in S} (\kappa_j-\alpha_j).
    \]
\end{lemma}
\begin{proof}
    See \cref{lem:bound_kappa_1_proof}.
\end{proof}

We now combine the lemmas above to prove our strongest capacity lower bound for polynomials in $\NHProd_n^d(\bm\alpha)$.

\begin{theorem} \label{thm:two_term_cap_bound}
    Fix any $\bm\kappa \in \Z_{\geq 0}^n$, $\bm\alpha \in \R_{\geq 0}^n$, and $\bm\epsilon, \bm\delta \in \R_{> 0}^n$ such that
    \[
        \max_{S \in \binom{[n]}{k}} \sum_{j \in S} (\alpha_j-\kappa_j) \leq 1 - \epsilon_k \qquad \text{and} \qquad \max_{S \in \binom{[n]}{k}} \sum_{j \in S} (\kappa_j-\alpha_j) \leq 1 - \delta_k
    \]
    for all $k \in [n]$. Then
    \[
        \cpc_{\bm\kappa}(p) \geq \min_{0 \leq \ell \leq n} \prod_{k=1}^\ell \epsilon_k \prod_{k=1}^{n-\ell} \delta_k
    \]
    for every $p \in \NHProd_n^d(\bm\alpha)$.
\end{theorem}
\begin{proof}
    We prove the desired result by induction on $(n,d)$ with lexicographical order. The base case is the case where $n=0$, which corresponds to $p \equiv 1$ for any $d$. In this case, $\bm\kappa$ is the empty vector and $\cpc_{\bm\kappa}(p) = 1$, which implies the desired result. (See \cref{sec:univ_case} for the $n=1$ case written out in detail.)

    For $n \geq 1$, we first handle the case that $\kappa_j = 0$ for some $j \in [n]$, and by permuting the variables we may assume that $\kappa_n = 0$. By \cref{lem:bound_kappa_0}, there exists $q \in \NHProd_{n-1}^d(\bm\beta)$ such that
    \[
        \cpc_{\bm\kappa}(p) \geq \epsilon_1 \cdot \cpc_{\bm\gamma}(q),
    \]
    where $\bm\gamma = (\kappa_1,\ldots,\kappa_{n-1}) \in \Z_{\geq 0}^{n-1}$ and $\bm\beta \in \R_{\geq 0}^{n-1}$ is such that for all $S \subseteq [n-1]$ we have
    \[
        \sum_{j \in S} (\alpha_j-\kappa_j) \leq \sum_{j \in S} (\beta_j-\gamma_j) \leq (\alpha_n-\kappa_n) + \sum_{j \in S} (\alpha_j-\kappa_j).
    \]
    Thus for all $k \in [n-1]$, we have
    \[
        \max_{S \in \binom{[n-1]}{k}} \sum_{j \in S} (\beta_j - \gamma_j) \leq (\alpha_n - \kappa_n) + \max_{S \in \binom{[n-1]}{k}} \sum_{j \in S} (\alpha_j - \kappa_j) \leq 1 - \epsilon_{k+1}
    \]
    and
    \[
        \max_{S \in \binom{[n-1]}{k}} \sum_{j \in S} (\gamma_j - \beta_j) \leq \max_{S \in \binom{[n-1]}{k}} \sum_{j \in S} (\kappa_j - \alpha_j) \leq 1 - \delta_k.
    \]
    Thus by induction we have that
    \[
        \cpc_{\bm\kappa}(p) \geq \epsilon_1 \cdot \cpc_{\bm\gamma}(q) \geq \min_{0 \leq \ell \leq n-1} \epsilon_1 \prod_{k=1}^\ell \epsilon_{k+1} \prod_{k=1}^{n-1-\ell} \delta_k = \min_{1 \leq \ell \leq n} \prod_{k=1}^\ell \epsilon_k \prod_{k=1}^{n-\ell} \delta_k,
    \]
    and this implies the desired result.
    
    Otherwise, $\kappa_j \geq 1$ for all $j \in [n]$. Now fix any extreme point $M \in \NHMat_n^d(\bm\alpha)$, and let $p$ be the polynomial associated to $M$. By \cref{lem:forest_extreme}, $M$ is the (weighted) bipartite adjacency matrix of a forest $G$. We call the vertices of $G$ corresponding to the rows of $M$ the \emph{left vertices} of $G$, and we call the vertices of $G$ corresponding to the columns of $M$ the \emph{right vertices} of $G$. Since $|\alpha_j - \kappa_j| < 1$ for all $j \in [n]$, we further have that none of the first $n$ columns of $M$ are equal to the zero vector, and thus $G$ has at most one (right) vertex of degree 0 (possibly the vertex corresponding to column $n+1$ of $M$).

    We first consider the case where $d \geq n+1$. Then $G$ has at least as many left vertices as right vertices. After possibly removing the right vertex in $G$ of degree 0, \cref{lem:left_leaves} implies $M$ has at least $d-n$ rows with exactly one non-zero entry. Letting $\bm{x}^{\bm{v}}$ be the polynomial associated to these rows of $M$, we have
    \[
        \cpc_{\bm\kappa}(p) = \inf_{\bm{x} > 0} \frac{\bm{x}^{\bm{v}} f(\bm{x})}{\bm{x}^{\bm\kappa}} = \cpc_{\bm\kappa-\bm{v}}(f).
    \]
    With this, we now redefine $\bm\kappa$ to be $\bm\kappa - \bm{v}$, $\bm\alpha$ to be $\bm\alpha - \bm{v}$, $p$ to be $f$, and $M$ to be the matrix with the (at least) $d-n$ rows removed. Since $(\alpha_j - v_j) - (\kappa_j - v_j) = \alpha_j - \kappa_j$ for all $j \in [n]$, the result in this case then follows by induction.

    We now consider the case where $d \leq n$. Then $G$ has at least one more right vertex than left vertices. After possibly removing the right vertex in $G$ of degree 0 (corresponding to column $n+1$ of $M$), \cref{lem:left_leaves} implies at least one of the first $n$ columns of $M$ has exactly one non-zero entry, and by permuting the variables we may assume that it is column $n$. Thus $\alpha_n \leq 1$, and $|\alpha_n - \kappa_n| < 1$ then implies $\kappa_n = 1$. By \cref{lem:bound_kappa_1}, there exists $q \in \NHProd_{n-1}^{d-1}(\bm\beta)$ such that
    \[
        \cpc_{\bm\kappa}(p) \geq \delta_1 \cdot \cpc_{\bm\gamma}(q),
    \]
    where $\bm\gamma = (\kappa_1,\ldots,\kappa_{n-1}) \in \Z_{\geq 0}^{n-1}$ and $\bm\beta \in \R_{\geq 0}^{n-1}$ is such that for all $S \subseteq [n-1]$ we have
    \[
        \sum_{j \in S} (\kappa_j-\alpha_j) \leq \sum_{j \in S} (\gamma_j-\beta_j) \leq (\kappa_n-\alpha_n) + \sum_{j \in S} (\kappa_j-\alpha_j).
    \]
    Thus for all $k \in [n-1]$, we have
    \[
        \max_{S \in \binom{[n-1]}{k}} \sum_{j \in S} (\beta_j - \gamma_j) \leq \max_{S \in \binom{[n-1]}{k}} \sum_{j \in S} (\alpha_j - \kappa_j) \leq 1 - \epsilon_k
    \]
    and
    \[
        \max_{S \in \binom{[n-1]}{k}} \sum_{j \in S} (\gamma_j - \beta_j) \leq (\kappa_n - \alpha_n) + \max_{S \in \binom{[n-1]}{k}} \sum_{j \in S} (\kappa_j - \alpha_j) \leq 1 - \delta_{k+1}.
    \]
    Thus by induction we have that
    \[
        \cpc_{\bm\kappa}(p) \geq \delta_1 \cdot \cpc_{\bm\gamma}(q) \geq \min_{0 \leq \ell \leq n-1} \delta_1 \prod_{k=1}^\ell \epsilon_k \prod_{k=1}^{n-1-\ell} \delta_{k+1} = \min_{0 \leq \ell \leq n-1} \prod_{k=1}^\ell \epsilon_k \prod_{k=1}^{n-\ell} \delta_k,
    \]
    and this implies the desired result.
    
    We have proven the result for polynomials associated to the extreme points of $\NHMat_n^n(\bm\alpha)$. Applying \cref{lem:log-concave_cap} then completes the proof.
\end{proof}

The next result gives the bound which we will use in \cref{sec:proof_app_TSP} to prove \cref{thm:main_imp_prob_bound}, the simply exponential improvement to the probability bound used for the metric TSP application.

\begin{corollary} \label{cor:one_term_cap_bound}
    Fix any $\bm\kappa \in \Z_{\geq 0}^n$ and $\bm\alpha \in \R_{\geq 0}^n$ such that $\bm\epsilon \in \R_{> 0}^n$ can be defined via
    \[
        \epsilon_k := 1 - \max_{\substack{S \subseteq [n] \\ |S| \leq k}} \left|\sum_{j \in S} (\kappa_j-\alpha_j)\right|
    \]
    for all $k \in [n]$. Then
    \[
        \cpc_{\bm\kappa}(p) \geq \prod_{k=1}^n \epsilon_k
    \]
    for every $p \in \NHProd_n^d(\bm\alpha)$.
\end{corollary}
\begin{proof}
    Note that
    \[
        \max_{\substack{S \subseteq [n] \\ |S| \leq k}} \sum_{j \in S} (\alpha_j - \kappa_j) \leq 1 - \epsilon_k \qquad \text{and} \qquad \max_{\substack{S \subseteq [n] \\ |S| \leq k}} \sum_{j \in S} (\kappa_j - \alpha_j) \leq 1 - \epsilon_k
    \]
    for all $k \in [n]$, and thus
    \[
        \cpc_{\bm\kappa}(p) \geq \min_{0 \leq \ell \leq n} \prod_{k=1}^\ell \epsilon_k \prod_{k=1}^{n-\ell} \epsilon_k
    \]
    by \cref{thm:two_term_cap_bound}. Since $\epsilon_1 \geq \epsilon_2 \geq \cdots \geq \epsilon_n$, this implies the result.
\end{proof}

By \cref{cor:prod-cap-bound}, the above results hold for all $p \in \NHStab_n^d(\bm\alpha)$, and we state this formally now.

\begin{corollary} \label{cor:strong_bounds_stable}
    \cref{thm:two_term_cap_bound} and \cref{cor:one_term_cap_bound} hold for all $p \in \NHStab_n^d(\bm\alpha)$.
\end{corollary}

\subsection{Proving \cref{thm:main_nh} and \cref{cor:main_nh_coeff}}

We now complete the proof of \cref{thm:main_nh}, and thus also of \cref{cor:main_nh_coeff}. Fix $p \in \NHStab_n^d(\bm\alpha)$. Thus
\[
    1 - \max_{\substack{S \subseteq [n] \\ |S| \leq k}} \left|\sum_{j \in S} (\kappa_j-\alpha_j)\right| \geq 1 - \|\bm\kappa - \bm\alpha\|_1
\]
and \cref{cor:strong_bounds_stable} (via \cref{cor:one_term_cap_bound}) imply
\[
    \cpc_{\bm\kappa}(p) \geq (1 - \|\bm\kappa - \bm\alpha\|_1)^n,
\]
which completes the proof. The tightness claim follows from \cref{lem:tightness}.

\subsection{Examples and Tightness} \label{sec:examples}

The $\delta$ parameters in \cref{thm:two_term_cap_bound} are tight, and this is shown in \cref{ex:delta_tight}. However, the $\epsilon$ parameters are not, as shown in \cref{ex:epsilon_tight}. \cref{ex:delta_tight} also demonstrates tightness of the dependence on the error parameter for some of our results, and we state this formally now.

\begin{lemma} \label{lem:tightness}
    The dependence on $(1 - \|\bm\kappa - \nabla p(\bm{1})\|_1)$ in \cref{thm:main_nh} and \cref{cor:main_nh_coeff} and the dependence on $\epsilon$ in \cref{thm:main_imp_prob_bound} are all tight for any fixed $\bm\kappa > \bm{0}$.
\end{lemma}
\begin{proof}
    Define $\alpha_1 := \kappa_1 - (1 - \epsilon) > 0$ and $\alpha_i := \kappa_i \geq 1$ for all $i \geq 2$. Let $p$ be the polynomial described by \cref{ex:delta_tight}, given explicitly by
    \[
        p(\bm{x}) = \left(\prod_{i=1}^n x_i^{\kappa_i-1}\right) \cdot \left(\prod_{i=1}^{n-1} (\epsilon x_i + (1-\epsilon) x_{i+1})\right) \cdot (\epsilon x_n + (1-\epsilon)).
    \]
    Then $\bm\kappa$ is a vertex of of the Newton polytope of $p$. Thus
    \[
        p_{\bm\kappa} = \cpc_{\bm\kappa}(p) = \prod_{k=1}^n \left(1 - \sum_{j=1}^k (\kappa_j - \alpha_j)\right) = \epsilon^n = (1 - \|\bm\kappa - \nabla p(\bm{1})\|_1)^n.
    \]
    To see that $p$ can be the probability generating polynomial for some random variables associated to a strongly Rayleigh distribution (as in \cref{thm:main_imp_prob_bound}), note that the polarization of $p$ is real stable and gives the probability generating polynomial for such a strongly Rayleigh distribution.
\end{proof}

\begin{example} \label{ex:delta_tight}
    Fix any $\bm\kappa \in \Z_{\geq 0}^n$ and $\bm\alpha \in \R_{\geq 0}^n$ such that $\kappa_j - \alpha_j \geq 0$ for all $j \in [n]$ and $\|\bm\kappa - \bm\alpha\|_1 < 1$. Thus $\kappa_j \geq 1$ for all $j \in [n]$. For $d = \|\bm\kappa\|_1$, consider the matrix $M = \left[\begin{smallmatrix} A \\ B \end{smallmatrix}\right]$, where $A$ is the $(\|\bm\kappa\|-n) \times (n+1)$ matrix given by
    \[
        A = \begin{bmatrix}
            \bm{1}_{\kappa_1-1} \bm{e}_1^\top \\
            \bm{1}_{\kappa_2-1} \bm{e}_2^\top \\
            \vdots \\
            \bm{1}_{\kappa_n-1} \bm{e}_n^\top
        \end{bmatrix},
    \]
    and $B$ is the $n \times (n+1)$ matrix for which
    \[
        b_{kk} = 1 - \sum_{j=1}^k (\kappa_j - \alpha_j), \qquad b_{k,k+1} = \sum_{j=1}^k (\kappa_j - \alpha_j),
    \]
    for all $k \in [n]$ and $b_{ij} = 0$ otherwise. Note that every row sum of $A$ is equal to 1, and the row sums of $B$ are given by
    \[
        \sum_{j=1}^{n} b_{kj} = 1 - \sum_{j=1}^k (\kappa_j - \alpha_j) + \sum_{j=1}^k (\kappa_j - \alpha_j) = 1
    \]
    for all $k \in [n]$. The column sums of $M$ are then given by
    \[
        \sum_{i=1}^d m_{ik} = (\kappa_k-1) + 1 - \sum_{j=1}^k (\kappa_j - \alpha_j) + \sum_{j=1}^{k-1} (\kappa_j - \alpha_j) = \kappa_k - (\kappa_k - \alpha_k) = \alpha_k
    \]
    for all $k \in [n]$. Thus $M \in \NHMat_n^d(\bm\alpha)$. Let $p$ be the polynomial associated to $M$, and let $q$ be the polynomial associated to $B$. We then have $\cpc_{\bm\kappa}(p) = \cpc_{\bm{1}}(q)$. Note that $\bm{1}$ is a vertex of the Newton polytope of $q$, and thus
    \[
        \cpc_{\bm\kappa}(p) = \cpc_{\bm{1}}(q) = \prod_{k=1}^n b_{kk} = \prod_{k=1}^n \left(1 - \sum_{j=1}^k (\kappa_j - \alpha_j)\right).
    \]
    By possibly permuting the variables to put $\kappa_j - \alpha_j$ in non-increasing order, this is precisely the lower bound guaranteed by \cref{thm:two_term_cap_bound}.
\end{example}

\begin{example} \label{ex:epsilon_tight}
    Consider the case that $\bm\kappa = \bm{0}$ and $\|\bm\alpha\|_1 < 1$. Given any $d$, let $M$ be any extreme point of $\NHMat_n^d(\bm\alpha)$. Then the column sum of column $n+1$ of $M$ is equal to $d-\sum_{j=1}^n \alpha_j > d-1$. Since every row sum equals $1$, all entries of column $n+1$ of $M$ are strictly positive. Since $M$ is an extreme point, this further implies that each column of $M$ has at most $1$ positive entry except column $n+1$. Letting $p$ be the polynomial associated to $M$, there exists a partition $S_1 \sqcup \cdots \sqcup S_k = [n]$ such that
    \[
        p(\bm{x}) = \prod_{i=1}^k \left(\left(\sum_{j \in S_i} \alpha_j x_j\right) + \left(1 - \sum_{j \in S_i} \alpha_j\right)\right).
    \]
    Thus by \cref{lem:prod_sum},
    \[
        \cpc_{\bm{0}}(p) = p_{\bm{0}} = \prod_{i=1}^k \left(1 - \sum_{j \in S_i} \alpha_j\right) \geq 1 - \sum_{j=1}^n \alpha_j.
    \]
    By \cref{lem:log-concave_cap}, this gives a lower bound on the capacity of every $p \in \NHProd_n^d(\bm\alpha)$. However, this lower bound is strictly better than the one guaranteed by \cref{thm:two_term_cap_bound}.
    
    As a note, this can be partially remedied by removing all $\kappa_j = 0$ columns at the same time (i.e., adjusting \cref{lem:bound_kappa_0} to remove many columns at once). However, it is currently unclear how to inductively do this correctly.
\end{example}

\section{Uniqueness of Permanent Minimizers} \label{sec:unique-minimizers}

In this section, let $\Mat_n(\bm{c})$ be the set of $n \times n$ matrices with non-negative entries and rows sums $\bm{1}$ and column sums $\bm{c} > 0$, let $p_M(\bm{x}) := \prod_{i=1}^n \sum_{j=1}^n m_{ij} x_j$ be the real stable polynomial associated to a given $M \in \Mat_n(\bm{c})$, and let $L(\bm{c})$ be any lower bound on $\cpc_{\bm{1}}(p_M)$ over all $M \in \Mat_n(\bm{c})$ (e.g., as given by \cref{thm:main_nh} above or the results of \cite{GL21}).

We now prove \cref{thm:unique-minimizers-main} via \cref{thm:unique-min} below. We note that the argument in the proof of \cref{thm:unique-min} given below can be made into a general statement about minimizers of quadratic forms. And further, the same argument given here applies to mixed discriminants, as mentioned in \cref{sec:main-results}.

\begin{lemma} \label{lem:strictly-pos}
    If
    \[
        \frac{(n-2)^{n-2} n^{n-1}}{(n-1)^{2n-3}} > \frac{c_1 c_2 \cdots c_n}{L(\bm{c})},
    \]
    then all minimizers of the permanent over over $\Mat_n(\bm{c})$ have all strictly positive entries.
\end{lemma}
\begin{proof}
    Note that the rank-one matrix $\frac{1}{n} \bm{1} \cdot \bm{c}^\top$ has all strictly positive entries and permanent equal to $\frac{n!}{n^n} \prod_{i=1}^n c_i$. Thus, to obtain a contradiction, let us assume that there exists $M \in \Mat_n(\bm{c})$ with at least one zero entry such that $\per(M) \leq \frac{n!}{n^n} \prod_{i=1}^n c_i$. By the main result of \cite{Gur07VdW}, we have
    \[
        \frac{n!}{n^n} \prod_{i=1}^n c_i \geq \per(M) \geq \left(\frac{n-2}{n-1}\right)^{n-2} \frac{(n-1)!}{(n-1)^{n-1}} \cdot L(\bm{c}),
    \]
    which after rearranging implies
    \[
        \frac{(n-2)^{n-2} n^{n-1}}{(n-1)^{2n-3}} > \left(\frac{n-2}{n-1}\right)^{n-2} \frac{n^{n-1}}{(n-1)^{n-1}},
    \]
    which is a contradiction.
\end{proof}


\begin{theorem} \label{thm:unique-min}
    If
    \[
        \frac{(n-2)^{n-2} n^{n-1}}{(n-1)^{2n-3}} > \frac{c_1 c_2 \cdots c_n}{L(\bm{c})},
    \]
    then the rank-one matrix $\frac{1}{n} \bm{1} \cdot \bm{c}^\top$ is the unique minimizer of the permanent over $\Mat_n(\bm{c})$.
\end{theorem}
\begin{proof}
    Let $M$ be a minimizer of the permanent over $\Mat_n(\bm{c})$. Thus every entry of $M$ is positive by \cref{lem:strictly-pos}. We will show that every pair of rows of $M$ must be equal, which immediately implies $M = \frac{1}{n} \bm{1} \cdot \bm{c}^\top$.
    
    Let $P \subset \R^{2 \times n}$ be the convex polytope given by the first two rows of all matrices in $\Mat_n(\bm{c})$ with rows $3$ through $n$ equal to those of $M$. Thus by positivity, $(\bm{m}_1,\bm{m}_2)$ is in the relative interior of $P$ where $\bm{m}_i$ is the $i^\text{th}$ row of $M$. Let $f(\bm{u},\bm{v})$ be defined as the permanent of the matrix $M$ with the first two rows replaced by $(\bm{u},\bm{v})$. Since $M$ minimizes the permanent and $(\bm{m}_1,\bm{m}_2)$ is contained in the relative interior of $P$, the necessary conditions on the minimum given by Lagrange multipliers implies
    \[
        \nabla f(\bm{m}_1,\bm{m}_2) = \left(\per(M_{(i,j)})\right)_{i \in [2], j \in [n]} = (a_i + b_j)_{i \in [2], j \in [n]}
    \]
    for some $a_i, b_j$, where $M_{(i,j)}$ is the matrix $M$ with row $i$ and column $j$ deleted. That is, the gradient of $f$ points in a direction orthogonal to $P$ at $(\bm{m}_1,\bm{m}_2)$.
    
    By row symmetry of the permanent, we also have that $\nabla f(\bm{m}_2, \bm{m}_1)  = (a_i' + b_j)_{i \in [2], j \in [n]}$, where $a_1' = a_2$ and $a_2' = a_1$. By row multilinearity of the permanent, we thus have
    \[
        \nabla f(t \cdot \bm{m}_1 + (1-t) \cdot \bm{m}_2, (1-t) \cdot \bm{m}_1 + t \cdot \bm{m}_2) = (t \cdot a_i + (1-t) \cdot a_i' + b_j)_{i \in [2], j \in [n]}
    \]
    for all $t \in \R$. If $\bm{m}_1 \neq \bm{m}_2$ then the gradient of $f$ is orthogonal to $P$ at all points on a line in $P$ through $M$, and thus the permanent is minimized at all these points. At least one such point (on the boundary of $P$) has a zero entry, which contradicts \cref{lem:strictly-pos}. Therefore it must be that $\bm{m}_1 = \bm{m}_2$. Applying this argument to every pair of rows of $M$ implies the desired result.
    %
\end{proof}

We now give the example which proves \cref{prop:minimizer-counter}.

\begin{example} \label{ex:min_per}
    Fix $t > 0$ and $n \in \N$, and define $\epsilon := \frac{1}{n^{1+t}}$. Further define $\bm\alpha \in \R_{>0}^n$ and $\bm{c} \in \R_{>0}^{n+1}$ via
    \[
        \bm\alpha := \big(1-\epsilon, 1-\epsilon, \ldots, 1-\epsilon) \in \R_{>0}^n \qquad \text{and} \qquad \bm{c} := \big(1+\alpha_1, \alpha_2, \alpha_3, \ldots, \alpha_n, \sum_{j=1}^n (1-\alpha_j)\big) \in \R_{>0}^{n+1}.
    \]
    Note that $\|\bm{c}-\bm{1}\|_1 = 1 + (n-2)\epsilon + 1 - n\epsilon = 2(1-\epsilon) < 2$. We first have
    \[
        \per\left(\frac{1}{n+1} \bm{1} \cdot \bm{c}^\top\right) = \frac{(n+1)!}{(n+1)^{n+1}} (2-\epsilon) (1-\epsilon)^{n-1} n\epsilon = \frac{n!}{n^{(1+t)n}} \cdot \frac{2-n^{-1-t}}{1-n^{-1-t}} \cdot \frac{(n^{1+t}-1)^n}{n^t (n+1)^n}.
    \]
    Now consider the matrix
    \[
        M = \begin{bmatrix}
            1 & 0 & 0 & \cdots & 0 & 0 \\
            1 - \sum_{j=1}^1 (1-\alpha_j) & \sum_{j=1}^1 (1-\alpha_j) & 0 & \cdots & 0 & 0 \\
            0 & 1 - \sum_{j=1}^2 (1-\alpha_j) & \sum_{j=1}^2 (1-\alpha_j) & \cdots & 0 & 0 \\
            0 & 0 & 1 - \sum_{j=1}^3 (1-\alpha_j) & \cdots & 0 & 0 \\
            \vdots & \vdots & \vdots & \ddots & \vdots & \vdots \\
            0 & 0 & 0 & \cdots & \sum_{j=1}^{n-1} (1-\alpha_j) & 0 \\
            0 & 0 & 0 & \cdots & 1 - \sum_{j=1}^n (1-\alpha_j) & \sum_{j=1}^n (1-\alpha_j)
        \end{bmatrix}.
    \]
    Note that $M \in \Mat_{n+1}(\bm{c})$. The matrix $M$ is upper-triangular, and thus we have
    \[
        \per(M) = \prod_{k=1}^n \sum_{j=1}^k (1-\alpha_j) = \prod_{k=1}^n (k\epsilon) = n! \cdot \epsilon^n = \frac{n!}{n^{(1+t)n}}.
    \]
    We then further have
    \[
        \frac{2-n^{-1-t}}{1-n^{-1-t}} \cdot \frac{(n^{1+t}-1)^n}{n^t (n+1)^n} \approx 2n^{(n-1)t} \cdot \frac{(n - n^{-t})^n}{(n+1)^n} \geq 2n^{(n-1)t} \cdot \frac{(n - 1)^n}{(n+1)^n} \approx \frac{2n^{(n-1)t}}{e^2} > 1,
    \]
    which implies $\per(M) < \per\left(\frac{1}{n+1} \bm{1} \cdot \bm{c}^\top\right)$ for large enough $n$.
\end{example}

\paragraph{Acknowledgements.} The third author would like to thank Vijay Bhattiprolu for many helpful and interesting discussions. The second author was supported by the National Science Foundation under Grant No. DMS-1926686. The third author acknowledges the support of the Natural Sciences and Engineering Research Council of Canada (NSERC), [funding reference number RGPIN-2023-03726]. Cette recherche a \'et\'e partiellement financ\'ee par le Conseil de recherches en sciences naturelles et en g\'enie du Canada (CRSNG), [num\'ero de r\'ef\'erence RGPIN-2023-03726].
Le troisi\`eme auteur remercie le Conseil de recherches en sciences naturelles et en g\'enie du Canada (CRSNG) de son soutien, [num\'ero de r\'ef\'erence RGPIN-2023-03726].

\printbibliography

\newpage

\appendix

\section{Improving the TSP Bound}\label{sec:app}

\subsection{Further Background on TSP}

In \cite{KKO21}, instead of only looking at vertices as we did in \cref{sec:proof_app_TSP}, they consider sets of edges going between two \textit{sets} of vertices. The notation $\bbe = {\bf (u,v)}$ is used, for $\bbe \subseteq E$, ${\bf u,v} \subseteq V$. Formally, if $\bbe = {\bf (u,v)}$ then  $\bbe = \{e=\{u,v\} \in E: |e \cap {\bf u}| = |e \cap {\bf v}| = 1\}$.
They call $\bbe$ a ``top edge bundle."\footnote{They also consider ``bottom edge bundles," which have a slightly different definition, however they do not arise in the lemmas we consider.} $\bbe={\bf (u,v)}$ is only defined for sets ${\bf u,v}$ for which $2 \le x(\delta({\bf u})) \le 2+\epsilon_\eta$ and $2 \le x(\delta({\bf v})) \le 2+\epsilon_\eta$, where $\epsilon_\eta \le 10^{-10}$ is a very small constant. 

In \cite{KKO21}, the event that the sets ${\bf u}$ and ${\bf v}$ are subtrees of $T$ is ubiquitous and therefore we always work in the conditional measure $\mu' = \mu_{\mid {\bf u,v} \text{ trees}}$. By Lemma 2.23 in \cite{KKO21}, ${\bf u,v}$ are subtrees with probability at least $1-2\epsilon_\eta$, $\mu'$ is a strongly Rayleigh distribution, and if $x'_e = \PP{T \sim \mu'}{e \in T}$ then $x(F) \le x'(F) \le x(F) + 2\epsilon_\eta$ for any $F \subseteq \delta({\bf u}) \cup \delta({\bf v})$.  Thus, since $\epsilon_\eta$ is such a small constant, there is very little difference between $\mu$ and $\mu'$.

For a top edge bundle $\bbe = {\bf (u,v)}$, they partition the edges of ${\bf u}$ into sets $A,B,C$. This is called the ``$(A,B,C)$ degree partition" of ${\bf u}$ (this is analogously done for ${\bf v}$). The details of this operation are not relevant in this work. 
What is relevant is that in the measure $\mu'$ we have the bounds
\begin{equation}	
\begin{aligned}
	x'(A),x'(B) \in [1-\eps_{1,1},1+3\eps_\eta],\\
	x'(C)\leq 2\eps_{1/1}+3\eps_\eta. 
\end{aligned}\label{eq:ABCDegParx}
\end{equation}
where $\eps_{1/1} = \epsilon_{1/2}/12$ and $\epsilon_{1/2} = 0.0002$ are small constants that will be relevant in this section.\footnote{Note in \cite{KKO21}, these bounds are given with respect to $x$, not $x'$. We use $x'$ for convenience since as noted everything is done in the conditional measure $\mu'$.} $\bbe$ is called a \textbf{half edge bundle} if $\frac{1}{2}-\epsilon_{1/2} \le x(\bbe) \le \frac{1}{2}+\epsilon_{1/2}$. As already observed in \cref{sec:proof_app_TSP}, these edge bundles are particularly difficult to deal with. Thus they are the subject of many probabilistic statements in \cite{KKO21}.

The final relevant definition is the event that a top edge bundle $\bbe={\bf (u,v)}$ is ``2-1-1 happy with respect to ${\bf u}$." Where for sets of edges $F,T \subseteq E$ we let $F_T = |F \cap T|$, this indicates that 
$$A_T=1,B_T=1,C_T=0,\delta({\bf v})_T=2, {\bf u,v} \text{ trees}$$
where $A,B$ and $C$ are the $(A,B,C)$ degree partition of ${\bf u}$. The event ``2-1-1 happy with respect to ${\bf v}$" is analogous, and also relevant, however since the assumptions on the partitions of ${\bf u}$ and ${\bf v}$ are the same we do not distinguish between them in this work.

The lemmas we improve all involve showing that the event that an edge is 2-1-1 happy (with respect to ${\bf u}$ or ${\bf v}$) has a non-negligible probability. 

We will require one additional fact about our spanning tree distribution, which is a simple consequence of negative association and homogeneity: 
\begin{lemma}[Explained in e.g. Section 2 of \cite{KKO21}]\label{lem:NAhomogeneity}
    Let $\mu:2^E \to \R_{\ge 0}$ be a strongly Rayleigh distribution over spanning trees with $x_e = \PP{T \sim \mu}{e \in T}$ as above. Let $A,B \subseteq E$ be disjoint sets. Then, $\mathbb{E}_{T \sim \mu}[A_T \mid B_T = 0] \le x(A) + x(B)$. Furthermore, the conditional measure $\mu_{\mid B_T=0}$ is strongly Rayleigh. Similarly, for an edge $e$, we can condition $e \in T$ and the measure $\mu_{\mid e \in T}$ is SR, and for any set $A \subseteq E$ with $e \not\in A$ we have $x(A) - x_e \le \EE{T \sim \mu}{A_T \mid e \in T} \le x(A)$. 
\end{lemma}

\subsection{Using Capacity to Improve Probabilistic Bounds} \label{sec:cap-prob-improvements}

\cref{thm:mainprobbound} and \cref{cor:mainprobbound_lite} can be applied straightforwardly to strengthen and simplify the proof of a number of lemmas in \cite{KKO21}. In this section we demonstrate this for Lemmas 5.21 and 5.22, as they are bottlenecks for the approximation factor. However, they can also be applied to simplify and improve the factor for Lemma 5.7 and achieve simpler proofs for Lemmas 5.16 and 5.17 at the price of a slightly worse constant. (Of course, with more effort, one can likely use these statements to strengthen the bounds for these as well similar to what is done for Lemma 5.24.) 

We then show that Lemma 5.23 can be easily strengthened using existing arguments. Finally, we use \cref{cor:mainprobbound_lite} to reduce Lemma 5.24 to a special case in which ad hoc methods can further improve the bound.


\subsubsection{Improving Lemma 5.21}

Here we improve upon the  bound of $0.005\eps_{1/2}^2$ via a straightforward application of our new capacity bound, \cref{cor:mainprobbound_lite}.
\begin{lemma}[Lemma 5.21 of \cite{KKO21}]\label{lem:new521}
Let $\bbe={\bf (u,v)}$ be a top edge bundle such that $x(\bbe)\leq 1/2-\eps_{1/2}$. If $\eps_{1/2}\leq 0.001$, $\epsilon_{1/1} \le \epsilon_{1/2}/12$, and $\epsilon_\eta \le \epsilon_{1/2}^2$ then $\bbe$ is 2-1-1 happy with probability at least $0.039\eps_{1/2}^2$.
\end{lemma}
\begin{proof}
Consider the $(A,B,C)$ degree partitioning of ${\bf u}$ (the case for ${\bf v}$ is exactly the same). We first condition on $\mu'$, $C_T = 0$ and $((A \cap C) \cup (B \cap C))_T = 0$. Let $V = \delta({\bf u})$. We now want to bound the probability that $A_T=B_T=1$ and $V_T=2$ in the resulting measure $\nu$. As we are working in the same setting as the proof of Lemma 5.21 in \cite{KKO21}, we re-use their bounds on the expectations using \cref{lem:NAhomogeneity} which ensures that they do not change too much:
\[
\begin{array}{rcccl}
    -0.5 & \leq & \E_\nu[A_T] - 1 & \leq & 0.5 \\
    -0.5 & \leq & \E_\nu[B_T] - 1 & \leq & 0.5 \\
    -0.5 & \leq & \E_\nu[V_T] - 2 & \leq & 0.01 \\
    -0.5 & \leq & \E_\nu[A_T + B_T] - 2 & \leq & 0.01 \\
    -1 + 1.8 \epsilon_{1/2} & \leq & \E_\nu[A_T + V_T] - 3 & \leq & 0.01 \\
    -1 + 1.8 \epsilon_{1/2} & \leq & \E_\nu[B_T + V_T] - 3 & \leq & 0.01 \\
    -1 + 1.75 \epsilon_{1/2} & \leq & \E_\nu[A_T + B_T + V_T] - 4 & \leq & -0.49 \\
\end{array}
\]
Now, note that since $V$ defines a cut in the graph and we have a distribution over spanning trees, $V_T \ge 1$ with probability 1. Therefore we can consider the random variable $V_T - 1$ instead (this corresponds to dividing the generating polynomial by the variable representing $V_T$). Thus in \cref{cor:mainprobbound_lite} we may choose (considering $\kappa=(1,1,1)$ on the variables $A_T,B_T,V_T-1$):
\[
\begin{array}{ccc}
    \epsilon_1 = 0.5, \qquad \epsilon_2 = 1.8 \epsilon_{1/2}, \qquad \epsilon_3 = 1.75 \epsilon_{1/2}
\end{array}
\]
which implies (using that $\nu$ occurs with probability at least $1/2$):
\[
    0.5 \cdot \Prob_\nu\left[A_T = 1, B_T = 1, V_T = 2\right] \geq 0.5 \cdot e^{-3} \cdot 0.5 \cdot 1.8 \epsilon_{1/2} \cdot 1.75 \epsilon_{1/2} \ge 0.0392 \cdot \epsilon_{1/2}^2.
\]
as desired.
\end{proof}

We also note that this bound is tight up to the constant factor for strongly Rayleigh distributions (where one can set $\gamma \approx \epsilon_{1/2}$):
\begin{lemma}
    There is a strongly Rayleigh distribution $\mu$ for which $x(A)=x(B)=1, x(V)=2$ and $x((A \cap V) \cup (B \cap V)) = \frac{1}{2}-\gamma$ such that 
    $$\PP{T \sim \mu}{A_T=1, B_T=1, V_T=2} = (1+2\gamma)\gamma^2$$
\end{lemma}
\begin{proof}
We will let $A$ be represented by $w$ and $x$, $B$ by $y$, and $V$ by $w$ and $z$.
Consider the real stable generating polynomial
    $$p_\mu(x,y,z,w) = z\left(\left(\frac{1}{2}+\gamma\right)+\left(\frac{1}{2}-\gamma\right)w\right)\left(\frac{1}{2}x+\frac{1}{2}z\right)\left(\gamma x+\gamma z + (1-2\gamma)y\right)((1-2\gamma)+2\gamma y)$$
One can check that the expectations are correct. Now we calculate the coefficient of $wyz$ and $xyz^2$, corresponding to the desired event. The distribution $\mu$ consists of 5 independent choices corresponding to each of the 5 terms. The first is of course trivial, we always take $z$. In the second, suppose we take $w$, so we have collected terms $wz$. Then, note we must take some variable from the third and fourth terms. So, $wyz$ has a coefficient of 0. 

Therefore we must take $(\frac{1}{2}+\gamma)$ in the second term, and we must now collect $xyz$ from the remainder. It follows that we must collect $y$ in the last term, as there are three terms and three variables, so we now have $(\frac{1}{2}+\gamma)2\gamma yz$ and must collect $xz$. There are two ways to do this, each with a coefficient of $\frac{\gamma}{2}$. Therefore the coefficient of $xyz^2$ is $(1+\gamma)\gamma^2$.
\end{proof}

\subsubsection{Improving Lemma 5.22}

We now improve upon Lemma 5.22 from \cite{KKO21}, which had a bound of $0.006\eps_{1/2}^2$. This is again a straightforward application of \cref{cor:mainprobbound_lite}.

\begin{lemma}[Lemma 5.22 of \cite{KKO21}]\label{lem:x_e>=1/2+eps_1/2}
Let $\bbe={\bf (u,v)}$ be a top edge bundle such that $x_\bbe\geq 1/2+\eps_{1/2}$. If $\eps_{1/2}\leq 0.001$, then, $\bbe$ is 2-1-1 happy with respect to $u$ with probability at least $0.038 \cdot \epsilon_{1/2}^2$.
\end{lemma}
\begin{proof}
    The setup of this proof is similar to \cref{lem:new521}. We again look at the $(A,B,C)$ degree partition of ${\bf u}$. First, we condition on $\mu'$ and $C_T=0$. Then, we condition on ${\bf u} \cup {\bf v}$ to be a tree. All these events occur with probability at least 1/2. Let $\nu$ be the resulting measure. Under $\nu$, notice that there is exactly one edge between ${\bf u}$ and ${\bf v}$, or in other words that $(A \cap \delta({\bf v})) \cup (B \cap \delta({\bf v}))$ is 1. 

    Therefore, to obtain $A_T=B_T=1$ and $\delta({\bf v})_T=2$, where $V=\delta({\bf v}) \smallsetminus \delta({\bf u})$ it suffices to prove that $A_T=B_T=1$ and $V_T=1$ (where note that $A,B,V$ are now disjoint).

    As this is the same setup as Lemma 5.22 in \cite{KKO21}, we simply re-use their bounds on the expectations they obtain with \cref{lem:NAhomogeneity}:
    \[
\begin{array}{rcccl}
    -0.5 & \leq & \E_\nu[A_T] - 1 & \leq & 0.5 \\
    -0.5 & \leq & \E_\nu[B_T] - 1 & \leq & 0.5 \\
    -0.005 & \leq & \E_\nu[V_T] - 1 & \leq & 0.5 \\
    -0.005 & \leq & \E_\nu[A_T + B_T] - 2 & \leq & 0.5 \\
    -0.01 & \leq & \E_\nu[A_T + V_T] - 2 & \leq & 1 - 1.75 \varepsilon_{1/2} \\
    -0.01 & \leq & \E_\nu[B_T + V_T] - 2 & \leq & 1 - 1.75 \varepsilon_{1/2} \\
    -0.01 & \leq & \E_\nu[A_T + B_T + V_T] - 3 & \leq & 1 - 1.75 \varepsilon_{1/2} \\
\end{array}
\]
Therefore in \cref{cor:mainprobbound_lite} we can set
\[
\begin{array}{ccc}
    \epsilon_1 = 0.5, \qquad \epsilon_2 = 1.75 \varepsilon_{1/2}, \qquad \epsilon_3 = 1.75 \varepsilon_{1/2} \\ 
\end{array}
\]
which implies
    $$0.5 \cdot \PP{\nu}{A_T = 1, B_T = 1, V_T = 1} \geq 0.5 \cdot e^{-3} \cdot 0.5 \cdot 1.75 \varepsilon_{1/2} \cdot 1.75 \varepsilon_{1/2} \ge  0.038 \cdot \varepsilon_{1/2}^2$$
as desired.
\end{proof}

\subsubsection{Improving Lemma 5.23}

First we recall some notation from \cite{KKO21}. For $\bbe={\bf (u,v)}$ and a set $A \subseteq E$, (where recall $\bbe \subseteq E$) we let $A_{-\bbe} = A \smallsetminus \bbe$ and $x_{\bbe(A)} = x(A \cap \bbe)$.

Unfortunately, this lemma does not easily fit into the framework from this paper, as it does not follow from expectation information. Instead, to improve Lemma 5.23, we first note that Lemma A.1 in \cite{KKO21} can be parameterized in the following way. We simply replace the constant 5 by $k$ in the statement and follow the previous proof.

\begin{restatable}{lemma}{alemma}\label{lem:A1}	
For a good half top edge bundle $\bbe={\bf (u,v)}$, let $A,B,C$ be the degree partitioning of $\delta(u)$, and  let $V=\delta({\bf v})_{-\bbe}$. If $\eps_{1/2}\leq 0.001$, $x_{\bbe(B)}\leq \eps_{1/2}$, and $\P{(A_{-\bbe})_T + V_T\leq 1}\geq k\eps_{1/2}$  then $\bbe$ is 2-1-1 good,
$$ \P{\bbe \text{ 2-1-1 happy w.r.t. $u$}} \geq 0.001(\min\{100,k\}-0.2)\eps_{1/2}^2$$
\end{restatable}

We defer the proof to \cref{app:furtherProb} since this section focuses on applications of our new machinery. Now we can show that in fact without any modifications, Lemma 5.23 is true with a much larger constant as long as one sets $\epsilon_{1/2} = 0.0002$. In particular:
 \begin{lemma}[Lemma 5.23 from \cite{KKO21}]
\label{lem:one-of-two-211}
Let $\bbe={\bf (v,u)}$ and $\bbf={\bf (v,w)}$ be good half top edge bundles and let $A,B,C$ be the degree partitioning of $\delta(v)$ such that $x_{\bbe(B)},x_{\bbf(B)}\leq \eps_{1/2}$.   
Then, if $\epsilon_{1/2} \le 0.0002$, one of $\bbe,\bbf$ is 2-1-1 happy with probability at least $0.0498\eps_{1/2}^2$.
\end{lemma}
\begin{proof}
    Let $U=\delta({\bf u}) \smallsetminus \bbe$ and $A_{-\bbe} = A \smallsetminus \bbe$. Then, Lemma 5.23 in \cite{KKO21} shows that $\P{U_T + (A_{-\bbe})_T \le 1} \ge 0.01$. Using that $\epsilon_{1/2} \le  0.0002$, this is at least $50\epsilon_{1/2}$. Now by \cref{lem:A1} we obtain a bound of $0.001(50-0.2)\epsilon_{1/2}^2$ as desired.
\end{proof}

\subsubsection{Improving Lemma 5.24}

Lemma 5.24 also does not follow directly from expectation information. However, we show that our new capacity bound is still helpful in improving the bound. We include the new part of the lemma here and defer the remainder of the proof to \cref{app:furtherProb} as it is similar to \cite{KKO21}.

\begin{lemma}[Lemma 5.24 from \cite{KKO21}]\label{lem:eABpositive211}
Let $\bbe={\bf (u,v)}$ be a good half edge bundle and 	let $A,B,C$ be the degree partitioning of $\delta(u)$.  If $\eps_{1/2}\leq 0.001$ and $x_{\bbe(A)},x_{\bbe(B)}\geq \eps_{1/2}$, 
 then
$$ \P{\bbe \text{ 2-1-1 happy w.r.t } u} \geq 0.0485\eps_{1/2}^2. $$
\end{lemma}
\begin{proof}
Condition $C_T$ to be zero, $u,v$ and $u\cup v$ be trees. Let $\nu$ be the resulting measure.
Let $X=A_{-\bbe}\cup B_{-\bbe}, Y=\delta(v)_{-\bbe}$. Then by \cref{lem:NAhomogeneity}:
\begin{align*}
&\EE{\nu}{X_T}, \EE{\nu}{Y_T}\in [1-3\eps_{1/1}, 1.5+\eps_{1/2}+2\eps_{1/1}+3\eps_{\eta}]\subset [0.995,1.51] \\
&\EE{\nu}{X_T+Y_T} \in [2.5 - 3\epsilon_{1/1} -3\epsilon_{1/2}, 3+2\epsilon_{1/2}+2\epsilon_{1/1}+4\epsilon_{\eta}] \in [2.495,3.01]
\end{align*}
Therefore using \cref{cor:mainprobbound_lite}, if $\EE{\nu}{X_T+Y_T} \le 2.999$, we could set $\epsilon_1 = 0.49,\epsilon_2=0.001$ and obtain a bound of $\PP{\nu}{X_T=Y_T=1} \ge e^{-2}(0.49)(0.001) \ge 0.06 \epsilon_{1/2}$. Therefore, we either have $\PP{\nu}{X_T=Y_T=1} \ge 0.06 \epsilon_{1/2}$, or it is the case that:
\begin{align*}
&\EE{\nu}{X_T}, \EE{\nu}{Y_T}\in [1.489,1.51], \qquad \EE{\nu}{X_T+Y_T} \in [2.999,3.01]
\end{align*}
Given this, we obtain the lemma using \cref{lem:eABpositive211-app}.
\end{proof}

\subsection{Further Probabilistic Statements}\label{app:furtherProb}

While in the previous section, we focused on the usage of \cref{cor:mainprobbound_lite}, here we require some of the ad hoc methods of \cite{KKO21} to make progress. 

A key fact about real stable polynomials we will use in the remaining statements is that their univariate restrictions are real rooted. As a consequence, the following standard fact can be shown:

\begin{lemma}\label{cor:rankseqBS}
Let $\mu:2^E \to \R_{\ge 0}$ be a strongly Rayleigh distribution and let $F \subseteq E$. Then  there exists independent Bernoulli random variables $B_1,\dots,B_m$ such that $\PP{T \sim \mu}{|F \cap T| = k} = \P{\sum_{i=1}^m B_i = k}$ for all $k \in \Z_{\ge 0}$.
\end{lemma}

Bernoulli random variables are quite easy to work with due to the following theorem of Hoeffding.
\begin{theorem}[Corollary 2.1 from \cite{Hoeff56}]\label{thm:hoeffding}
Let $g:\{0,1,\dots,n\}\to \R$ and $0\leq q\leq n$ for some integer $n\geq 0$.  Let $B_1,\dots,B_n$ be $n$ independent Bernoulli random variables with success probabilities $p_1,\dots,p_n$, where $\sum_{i=1}^n p_n = q$ that minimizes (or maximizes)
$$ \E[g(B_1+\dots+B_n)]$$
over all such distributions. Then,  $p_1,\dots,p_n\in\{0,x,1\}$ for some $0<x<1$. In particular, if only $m$ of $p_i$'s are nonzero and $\ell$ of $p_i$'s are 1, then the remaining $m-\ell$ are $\frac{q-\ell}{m-\ell}$.
\end{theorem} 
In this section we will apply \cref{thm:hoeffding} in the following manner. First we use \cref{cor:rankseqBS} to show that the random variable $A=|F \cap T|$ for some $F \subseteq E$ is distributed as a sum of independent Bernoullis. Then we apply \cref{thm:hoeffding} to this distribution with a function $g$ that indicates if $A$ has a certain size or a certain parity. For example we may be interested in a lower bound on the probability that $A \ge 1$ given that its expectation is 1. \cref{thm:hoeffding} easily allows us to show that this probability is at least $1/e$ since in the worst case all $n$ non-zero Bernoullis are equal, giving a probability of $(1-\frac{1}{n})^{n-1} \ge 1/e$. See Lemma 2.21 of \cite{KKO21} for a generic lower bound on such probabilities using \cref{thm:hoeffding}.

Finally, we will need the following two statements from \cite{KKO21}.

\begin{lemma}[Lemma 5.4 from \cite{KKO21}]\label{lem:updowntruncation}
Given a strongly Rayleigh distribution $\mu:2^{[n]}\to \R_{\geq 0}$, let $A,B$ be two (nonnegative) random variables corresponding to the number of elements sampled from two {\em disjoint} sets such that  $\P{A+B=n}>0$ where $n=n_A+n_B$. Then,
\begin{eqnarray}
	\P{A\geq n_A | A+B=n}=\P{B\leq n_B | A+B=n} & \geq & \P{A\geq n_A}\P{B\leq n_B},  \\
	\P{A\leq n_A|A+B=n}=\P{B\geq n_B | A+B=n}  &\geq &  \P{A\leq n_A}\P{B\geq n_B}.
\end{eqnarray} 
\end{lemma}

\begin{corollary}[Corollary 5.5 from \cite{KKO21}]\label{lem:SRA=nA}
Let $\mu:2^{[n]}\to\R_{\geq 0}$ be a SR distribution. Let $A,B$ be two random variables corresponding to the number of elements sampled from two disjoint sets of elements such that $A \ge k_A$ with probability 1 and $B \ge k_B$ with probability 1. If $\P{A\geq n_A},\P{B\geq n_B}\geq \epsilon_1$ and $\P{A\leq n_A},\P{B\leq n_B}\geq \epsilon_2$, then, letting $n'_A = n_A-k_A, n'_B = n_B-k_B$,
\begin{align*}
	&\P{A=n_A | A+B=n_A+n_B}\geq \epsilon \min\{\frac{1}{n'_A+1},\frac{1}{n'_B+1}\},\\
&\P{A=n_A | A+B=n_A+n_B}\geq \min\left\{p_m, \epsilon(1-(\epsilon /p_m)^{1/\max\{n'_A,n'_B\}})\right\}
\end{align*}
where $\epsilon=\epsilon_1\epsilon_2$ and $p_m\leq\max_{k_A\leq k\leq n_A+n_B-k_B} \P{A=k | A+B=n_A+n_B}$ is a lower bound on the mode of $A$.
\end{corollary}

The following proof is essentially identical to the the proof of Lemma A.1 in \cite{KKO21}, we simply parameterize the statement by $k$.
\alemma*
\begin{proof}
We first condition on $\mu'$, $C_T=0$, $u\cup v$  to be a tree and let $\nu$ be the resulting SR measure on edges in $A,B,V$. Since $\bbe$ is 2-2 good, by Lemma 5.15 in \cite{KKO21} and negative association,
\begin{align*}
\PP{\nu}{(\delta(u)_{-\bbe})_T+V_T\leq 2} \geq \P{(\delta(u)_{-\bbe})_T+V_T\leq 2} - \P{C_T=0} \geq 0.4\eps_{1/2} - 2\eps_{1/1}-\eps_\eta\geq 0.22\eps_{1/2},	
\end{align*}
where we used $\eps_{1/1}\le \eps_{1/2}/12$.
Letting $p_i = \P{(\delta(u)_{-\bbe})_T+V_T = i}$, we therefore have $p_{\le 2} \ge 0.22 \eps_{1/2}$. In addition, by
\cref{thm:hoeffding}, $p_3 \ge 1/4$. 
If $p_2 < 0.2 \eps_{1/2}$, then from
$p_2/p_3 \le 0.8 \eps_{1/2}$, we could use log-concavity to derive a contradiction to $p_{\le 2} \ge 0.22 \eps_{1/2}$.
Therefore, we must have
$$\PP{\nu}{A_T+B_T+V_T=3}=\PP{\nu}{(\delta(u)_{-\bbe})_T+V_T=2}\geq 0.2\eps_{1/2}.$$

Next, notice since $\P{u,v,u\cup v\text{ trees}, C_T=0}\geq 0.49$, by the lemma's assumption, $\PP{\nu}{\bbe(B)}\leq 2.01\eps_{1/2}$. Therefore, 
$$\EE{\nu}{B_T+V_T}\leq x(V)+x(B)+1.01\eps_{1/2}+2\eps_{1/1}+\eps_\eta \leq 2.51.$$
So, by Markov, $\PP{\nu}{B_T+V_T\leq 2}\geq 0.15$. 
Finally, by negative association, 
$$\PP{\nu}{A_T+V_T\leq 2} \geq \PP{\nu}{(A_{-\bbe})_T+V_T\leq 1} \geq \P{(A_{-\bbe})_T+V_T\leq 1} - \P{C_T=0} \geq (k-0.2)\eps_{1/2}$$
where we used the lemma's assumption. 

Using arguments from the proof of Lemma 5.21 in \cite{KKO21}, we have $$\PP{\nu}{A_T+B_T=2 \mid A_T+B_T+V_T=3} \geq 0.12$$

Now by \cref{lem:updowntruncation}, 
\begin{align*}
    \PP{\nu}{B_T \ge 1 \mid A_T+B_T+V_T = 3} &= \PP{\nu}{A_T+B_T \le 2 \mid A_T+B_T+V_T=3} \\
    &\ge \PP{\nu}{A_T+B_T \le 2}\PP{\nu}{V_T \ge 1} \ge 0.63(k-0.2)\eps_{1/2}
\end{align*}
using that similar to Lemma 5.21 in \cite{KKO21} we can show $\PP{\nu}{V_T \ge 1} \ge 0.63$. We also use similar arguments to show $\PP{\nu}{B_T \le 1 \mid A_T+B_T+V_T=3} \ge 0.147$.

Therefore, by \cref{lem:SRA=nA}, $\PP{\nu}{B_T=1 \mid A_T+B_T=2,V_T=1} \geq 0.147 \cdot 0.63(k-0.2)\eps_{1/2} \cdot p_m$ where $p_m \le \max_{k \in \{0,1,2\}}\P{B=k \mid A_T+B_T=2,V_T=1}$. For $k \le 100$ and $\epsilon_{1/2} \le 0.0002$, our bound for $\PP{\nu}{B_T=1 \mid A_T+B_T=2,V_T=1}$ is at most 0.001 for $p_m =1$. Therefore, we may assume it is at most 0.001. By log concavity we can set $p_m = 0.998$. So,
$$
	\P{\bbe \text{ 2-1-1 happy}} \geq (0.092(\max\{k,100\}-0.2)\eps_{1/2})0.12(0.2\eps_{1/2})0.498\geq 0.001(\max\{k,100\}-0.2)\eps_{1/2}^2
$$
as desired.
\end{proof}

Here we use the additional assumptions provided by \cref{lem:eABpositive211} to improve upon the original bound.
\begin{lemma}\label{lem:eABpositive211-app}
Let $\bbe={\bf (u,v)}$ be a good half edge bundle and 	let $A,B,C$ be the degree partitioning of $\delta(u)$.  Suppose $\eps_{1/2}\leq 0.001$ and $x_{\bbe(A)},x_{\bbe(B)}\geq \eps_{1/2}$. Let $\nu$ be the measure conditioned on $C_T=0$, $u,v$ and $u\cup v$ be trees. Suppose that either $\PP{\nu}{X_T=Y_T=1} \ge 0.06\epsilon_{1/2}$, or 
\begin{align*}
&\EE{\nu}{X_T}, \EE{\nu}{Y_T}\in [1.489,1.51], \qquad \EE{\nu}{X_T+Y_T} \in [2.999,3.01]
\end{align*}
Then,
$$ \P{\bbe \text{ 2-1-1 happy w.r.t } u} \geq 0.0485\eps_{1/2}^2. $$
\end{lemma}
\begin{proof}
We break the proof into two stages.
\paragraph{First we show the bound given $\PP{\nu}{X_T=Y_T=1} \ge 0.0498\epsilon_{1/2}$.} Let ${\cal E}$ be the event $\{X_T=Y_T=1\}$ in the conditional measure $\nu$.
Note that in $\nu$ we always choose exactly 1 edge from the $\bbe$ bundle and that is independent of edges in $X,Y$, in particular the above event. Therefore, we can  correct the parity of $A,B$ by choosing from $e_A$ or $e_B$.
It follows that
$$ \P{\bbe \text{ 2-1-1 happy w.r.t $u$}} \geq \PP{\nu}{{\cal E}}(1.99\eps_{1/2}) 0.49 \geq 0.0485\eps_{1/2}^2,$$
where we used that the probability $C_T=0$ and $u,v,u \cup v$ are subtrees with probability at least 0.49, $\EE{\nu}{\bbe(A)_T}  \geq 1.99\eps_{1/2}$, and $\EE{\nu}{\bbe(B)_T}  \geq 1.99\eps_{1/2}$. To see why these latter inequalities hold, observe that conditioned on $u,v$ trees, we always sample at most one edge between $u,v$. Therefore, since under $\nu$ we choose exactly one edge between $u,v$, the probability of choosing from $e(A)$ (and similarly choosing from $\bbe(B)$) is at least 
$$\frac{\EE{}{\bbe(A)_T \mid u,v\text{ trees}, C_T=0}}{\P{\bbe \mid u,v\text{ trees}, C_T=0}} \geq \frac{x_{\bbe(A)}-2\eps_\eta}{x_{\bbe}+3\eps_{1/1}}\geq \frac{\eps_{1/2}-2\eps_{\eta}}{1/2+1.3\eps_{1/2}}\geq 1.99\eps_{1/2}$$
as desired.

\paragraph{Second we show $\PP{\nu}{X_T=Y_T=1} \ge 0.0498\epsilon_{1/2}$.} If this is true by the lemma's assumption, we are done, otherwise we have the bounds on the expectations:
\begin{align*}
&\EE{\nu}{X_T}, \EE{\nu}{Y_T}\in [1.489,1.51], \qquad \EE{\nu}{X_T+Y_T} \in [2.999,3.01]
\end{align*}

First suppose that $\PP{\nu}{X_T+Y_T=2} \ge 0.001 \ge \epsilon_{1/2}$. Applying \cref{lem:SRA=nA} (and using \cref{thm:hoeffding} to bound the probability $X_T,Y_T$ are at least 1 and at most 1), we obtain $\PP{\nu}{X_T= 1 | X_T+Y_T=2}\geq \frac{1}{2}(0.245)(0.77) \ge 0.094$ and therefore again our desired bound $\PP{\nu}{X_T=Y_T=1} \ge 0.06 \epsilon_{1/2}$ would hold. Therefore we may assume that $\PP{\nu}{X_T+Y_T=2} \le 0.001$. By log concavity and the fact that $\EE{\nu}{X_T+Y_T} \ge 2.999$, it must be that $\PP{\nu}{X_T+Y_T=3} \ge 0.995$.

We use this information to improve the bounds on $\PP{\nu}{X_T\geq 1}, \PP{\nu}{Y_T\geq 1}$ and $\PP{\nu}{X_T\leq 1}, \PP{\nu}{Y_T\leq 1}$. First, condition on $X_T+Y_T=3$, this occurs with probability at least 0.995. Call the resulting measure $\nu'$. By log concavity, $$1.48 \le \EE{\nu'}{X_T},\EE{\nu'}{Y_T} \le 1.52$$
Finally we apply \cref{thm:hoeffding} and use the fact that there are at most 3 Bernoullis. This demonstrates that
\begin{align*}
&\PP{\nu}{X_T\geq 1}, \PP{\nu}{Y_T\geq 1}\geq 0.864 \\
&\PP{\nu}{X_T\leq 1}, \PP{\nu}{Y_T\leq 1}\geq 0.419
\end{align*}
Where we use that in $\nu'$ the quantities are bounded by $0.869$ and $0.422$ respectively. We now turn to the fact that $\bbe$ is 2-2 good to bound the probability the sum is 2. By Lemma 5.15 from \cite{KKO21} and stochastic dominance, 
$$\PP{\nu}{X_T + Y_T\leq 2}\geq \P{(\delta(u)_{-\bbe})_T+Y_T\leq 2} - \P{C_T=0} \geq 0.4\eps_{1/2} - 2\eps_{1/1}-\eps_\eta\geq 0.22\eps_{1/2},$$
where we used $\eps_{1/1}<\eps_{1/2}/12$.
It follows by log-concavity of $X_T+Y_T$ and the fact that $X_T+Y_T \ge 1$ with probability 1 under $\nu$ (as we sample a tree) that $\PP{\nu}{X_T+Y_T=2}\geq 0.2\eps_{1/2}$. Finally, we obtain $\PP{\nu}{X_T=1 \mid X_T+Y_T=2} \ge 0.249$. This is because of the following. WLOG, suppose we have $\EE{\nu}{X_T \mid X_T+Y_T=2} \le 1$ (otherwise we bound $\PP{\nu}{Y_T=1 \mid X_T+Y_T=2}$). Now by \cref{lem:updowntruncation}, we have that $\PP{\nu}{X_T \ge 1 \mid X_T+Y_T=2} \ge 0.362$. If $\PP{\nu}{X_T = 1 \mid X_T+Y_T=2} \ge 0.25$ we are done already, therefore $\PP{\nu}{X_T=2 \mid X_T+Y_T=2} \ge 0.112$. This implies that $\EE{\nu}{X_T \mid X_T+Y_T=2} \ge 0.474$, as:
$$\EE{\nu}{X_T \mid X_T+Y_T=2} = \PP{\nu}{X_T \ge 1 \mid X_T+Y_T=2} + \PP{\nu}{X_T=2 \mid X_T+Y_T=2} \ge 0.474$$
Now applying \cref{thm:hoeffding} with at most two Bernoullis, we obtain the bound. 

Therefore, $\PP{\nu}{X_T=Y_T=1} \ge (0.2\eps_{1/2})(0.249) = 0.0498\epsilon_{1/2}$ as desired.
\end{proof}

\subsection{A Parameterized TSP Bound}

The following statement is an immediate consequence of \cite{KKO21,KKO22}:
\TSPblackbox*
To see this, one can examine the proof of the ``main payment theorem," Theorem 4.33, of \cite{KKO21}. As it is a lengthy statement with many definitions not given here, we refer the reader to \cite{KKO21}. In the proof of this theorem, there is a constant $\epsilon_P$ is set to be $1.56 \cdot 10^{-6}p$, where $p$ is the minimum probability guaranteed by Corollary 5.9, Lemma 5.25, and Lemma 5.28. Then, Lemma 5.25 and Lemma 5.28 are simply aggregations of the bounds in Lemmas 5.21, 5.22, 5.23, 5.24, and 5.27, which are the lemmas we improve in this paper. It is also required that $p \le 3\epsilon_{1/2}$ and $\epsilon_{1/2} \le 0.0002$, as this the threshold for the definition of a good edge from Definition 5.13, which is related to Lemma 5.16 and Lemma 5.17. Thus it is sufficient that $p \le 10^{-4}$ if one does not modify these lemmas. 

Therefore, given the assumptions of the lemma, Theorem 4.33 holds without any modification of the proof for $\epsilon_P = 1.56 \cdot 10^{-6}p$. We can then plug the improved version of the payment theorem into Theorem 6.1 from \cite{KKO22}, where one sets $\eta = \min\{10^{-12}, \frac{\epsilon_P}{750}\}$, $\beta = \frac{\eta}{4+2\eta}$ and achieves an approximation ratio (and integrality gap) of 
$$\frac{3}{2}-\frac{\epsilon_P}{6}\beta+\frac{\epsilon_P\eta}{100} \ge \frac{3}{2}-\frac{\epsilon_P^2}{25000} \ge \frac{3}{2}-9.7 p^2\cdot 10^{-17}$$ 
As long as $\epsilon_P/750 < 10^{-12}$ which it is for $p \le 10^{-4}$.


\section{Omitted Proofs from \cref{sec:main_proofs}}

\subsection{Proof of \cref{lem:log-concave_cap}} \label{lem:log-concave_cap_proof}

Fix any $A,B \in \NHMat_n^d(\bm\alpha)$, and let $p,q,f$ be the polynomials associated to $A,B, \frac{A+B}{2}$ respectively. By the AM-GM inequality, we compute
\[
\begin{split}
    \cpc_{\bm\kappa}(f) &= \inf_{\bm{x} > 0} \frac{\prod_{i=1}^d \left(\frac{a_{i,n+1} + b_{i,n+1}}{2} + \sum_{j=1}^n \frac{a_{i,j} + b_{i,j}}{2} \cdot x_j\right)}{\bm{x}^{\bm\kappa}} \\
        &= \inf_{\bm{x} > 0} \frac{\prod_{i=1}^d \left[\frac{\left(a_{i,n+1} + \sum_{j=1}^n a_{i,j} x_j\right) + \left(b_{i,n+1} + \sum_{j=1}^n b_{i,j} x_j\right)}{2}\right]}{\bm{x}^{\bm\kappa}} \\
        &\geq \inf_{\bm{x} > 0} \left[\frac{\prod_{i=1}^d \left(a_{i,n+1} + \sum_{j=1}^n a_{i,j} x_j\right)}{\bm{x}^{\bm\kappa}} \cdot \frac{\prod_{i=1}^d \left(b_{i,n+1} + \sum_{j=1}^n b_{i,j} x_j\right)}{\bm{x}^{\bm\kappa}}\right]^{\frac{1}{2}} \\
        &\geq \left[\cpc_{\bm\kappa}(p) \cdot \cpc_{\bm\kappa}(q)\right]^{\frac{1}{2}}
\end{split}
\]

\subsection{Proof of \cref{lem:left_leaves}} \label{lem:left_leaves_proof}

We prove this by induction, where the base case is any path graph. For this case, we have $m \in \{n-1,n,n+1\}$. If $m = n-1$ then $G$ has 0 left leaves, if $m = n$ then $G$ has 1 left leaf, and if $m = n+1$ then $G$ has two left leaves. Thus the desired result holds in this case.

For the inductive step, $G$ is not a path graph. Let $v$ be any leaf of $G$, and construct a new graph $G'$ as follows. Let $v_0 := v$ and remove $v_0$ from $G_0 := G_0$ to create the graph $G_1$. Let $v_1$ be the one neighbor of $v_0$ in $G$. If $v_1$ is a leaf or vertex of degree 0 in $G_1$, then remove $v_1$ from $G_1$ to create the graph $G_2$. If $v_1$ was a vertex of degree 0, then stop and define $G' := G_1$. Otherwise let $v_2$ be the one neighbor of $v_1$, and continue this process inductively until $v_k$ is not a leaf or a vertex of degree 0 in $G_k$. Once the process stops, define $G' := G_k$. Note that $G'$ is a bipartite forest which has no vertices of degree 0, and $G'$ is non-empty since $G$ is not a path graph. We now have two cases: $v$ is a left leaf of $G$, or $v$ is a right leaf of $G$.

First suppose $v$ is a left leaf of $G$. Then for some $i$, $G'$ has $m-i$ left vertices and at most $n-i+1$ right vertices. Thus by induction, the number of left leaves of $G$ is at least
\[
    1 + (m-i) - (n-i+1) + 1 = m - n + 1
\]
since $v$ is a left leaf. Thus the result holds in this case.

Next suppose $v$ is a right leaf of $G$. Then for some $i$, $G'$ has $m-i$ left vertices and at most $n-i$ right vertices. Thus by induction, the number of left leaves of $G$ is at least
\[
    (m-i) - (n-i) + 1 = m - n + 1.
\]
Thus the result holds in this case as well.

\subsection{Proof of \cref{lem:prod_sum}} \label{lem:prod_sum_proof}

We prove the desired result by induction, where the case of $d=1$ is trivial. Now note that
\[
    \frac{1-\sum_{i=1}^d c_i}{1-c_d} = \frac{(1-c_d)\left(1 - \sum_{i=1}^{d-1} c_i\right) + c_d\left(1 - \sum_{i=1}^{d-1} c_i\right) - c_d}{1-c_d} = 1 - \sum_{i=1}^{d-1} c_i - \frac{c_d}{1-c_d} \sum_{i=1}^{d-1} c_i \leq 1 - \sum_{i=1}^{d-1} c_i.
\]
Thus by induction we have
\[
    1-\sum_{i=1}^d c_i \leq (1-c_d) \left(1 - \sum_{i=1}^{d-1} c_i\right) \leq (1-c_d) \prod_{i=1}^{d-1} (1-c_i),
\]
and this completes the proof.

\subsection{Proof of \cref{lem:bound_kappa_0}} \label{lem:bound_kappa_0_proof}

Define $q_0(\bm{x}) = q_0(x_1,\ldots,x_{n-1})$ via
\[
    q_0(\bm{x}) = p(x_1,\ldots,x_{n-1},0),
\]
so that
\[
    \cpc_{\bm\kappa}(p) = \cpc_{\bm\gamma}(q_0).
\]
Note that $q_0$ is the polynomial associated to a $d \times n$ matrix $A$, where the row sums of $A$ are given by $1 - m_{in}$ for all $i$. Note that $m_{in} \leq \alpha_n \leq 1-\epsilon < 1$ implies $1 - m_{in} > 0$ for all $i$. We now construct a new matrix $B \in \NHMat_{n-1}^d(\bm\beta)$ by dividing row $i$ of $A$ by $1 - m_{in}$ for all $i$. Defining $q(\bm{x})$ to be the polynomial associated to the matrix $B$, we have
\[
    \cpc_{\bm\kappa}(p) = \cpc_{\bm\gamma}(q_0) = \prod_{i=1}^d (1 - m_{in}) \cdot \cpc_{\bm\gamma}(q) \geq \epsilon \cdot \cpc_{\bm\gamma}(q)
\]
by \cref{lem:prod_sum}. Finally, for all $S \subseteq [n-1]$ we compute
\[
\begin{split}
    \sum_{j \in S} (\beta_j - \gamma_j) &= \sum_{j \in S} \left(\alpha_j - \kappa_j + \sum_{i=1}^d \left(\frac{m_{ij}}{1-m_{in}} - m_{ij}\right)\right) \\
        &= \sum_{i=1}^d \frac{m_{in}}{1-m_{in}} \sum_{j \in S} m_{ij} + \sum_{j \in S} (\alpha_j-\kappa_j) \\
        &\leq \sum_{i=1}^d \frac{m_{in}}{1-m_{in}} (1-m_{in}) + \sum_{j \in S} (\alpha_j-\kappa_j) \\
        &= (\alpha_n-\kappa_n) + \sum_{j \in S} (\alpha_j-\kappa_j).
\end{split}
\]
The other inequality then follows from the fact that $\sum_{i=1}^d \frac{m_{in}}{1-m_{in}} \sum_{j \in S} m_{ij} \geq 0$.

\subsection{Proof of \cref{lem:bound_kappa_1}} \label{lem:bound_kappa_1_proof}

Let $i_0$ be the row index of the one non-zero entry in column $n$. Note that $\epsilon \leq \alpha_n = m_{i_0,n} \leq 1$. Define $q(\bm{x}) = q(x_1,\ldots,x_{n-1})$ via
\[
    p(\bm{x}) = \left(m_{i_0,n+1} + \sum_{j=1}^n m_{i_0,j} x_j\right) \cdot q(\bm{x}).
\]
That is, $q(\bm{x})$ is the polynomial associated to the matrix obtained by removing row $i_0$ of $M$. We then have
\[
    \cpc_{\bm\kappa}(p) \geq \cpc_{\bm{e}_n}\left(m_{i_0,n+1} + \sum_{j=1}^n m_{i_0,j} x_j\right) \cdot \cpc_{\bm\gamma}(q) = m_{i_0,n} \cdot \cpc_{\bm\gamma}(q) \geq \epsilon \cdot \cpc_{\bm\gamma}(q).
\]
Note that $q$ is the polynomial associated to a $(d-1) \times n$ matrix $A$ with row sums all equal to 1. Let $\bm\beta$ be the vector of the first $n-1$ column sums of this matrix, so that $q \in \NHProd_{n-1}^{d-1}(\bm\beta)$. For all $S \subseteq [n-1]$ we have
\[
\begin{split}
    \sum_{j \in S} (\gamma_j-\beta_j) &= \sum_{j \in S} (\kappa_j-\alpha_j+m_{i_0,j}) \\
    &= \sum_{j \in S} m_{i_0,j} + \sum_{j \in S} (\kappa_j-\alpha_j) \\
        &\leq (1-m_{i_0,n}) + \sum_{j \in S} (\kappa_j-\alpha_j) \\
        &= (\kappa_n-\alpha_n) + \sum_{j \in S} (\kappa_j-\alpha_j).
\end{split}
\]
The other inequality then follows from the fact that $\sum_{j \in S} m_{i_0,j} \geq 0$.

\end{document}